\documentclass[11pt]{amsart}
\usepackage{amsmath,amssymb,latexsym,cancel,rotating}
\usepackage{diagrams}

\newtheorem{theorem}{Theorem}[section]
\newtheorem{lemma}[theorem]{Lemma}
\newtheorem{corollary}[theorem]{Corollary}

\newtheorem{proposition}[theorem]{Proposition}

\newtheorem{remark}[theorem]{Remark}

\def\cP{{\mathcal P}} 
 \def\cD{{\mathcal D}}     \def\rE{{\textrm E}} \def\rG{{\textrm G}}     \def\cN{{\mathcal N}}
     
 \def\cO{{\mathcal O}} 
    \def\cQ{{\mathcal Q}}  
    
\def\cK{{\mathcal K}} \def\cF{{\mathcal F}}
\def\bbN{{\mathbb N}}  \def\bbZ{{\mathbb Z}}  \def\bbQ{{\mathbb Q}}
    \def\bbF{{\mathbb F}}
\def\bbC{{\mathbb C}}

\def\Hom{\mbox{\rm Hom}}  
       
\def\Ext{\mbox{\rm Ext}\,}   
\def\dim{\mbox{\rm dim}\,}   
    \def\Ind{\mbox{\rm Ind}}  \def\Res{\mbox{\rm Res}}\def\ind{\mbox{\rm ind}}\def\res{\mbox{\rm res}}

\def\lp{{}^p}

\def\bnu{{|\nu\rangle}}

\def\bnu{{\boldsymbol{\nu}}}

\begin{document}

\title[The parity of Lusztig's restriction functor and Green's formula]
{The parity of Lusztig's restriction functor and Green's formula}
\thanks{Jiepeng Fang and Yixin Lan were supported by Tsinghua University Initiative Scientific Research Program (No. 2019Z07L01006), and Jie Xiao was supported by NSF of China (No. 12031007).}

\author[Jiepeng Fang, Yixin Lan, Jie Xiao]{Jiepeng Fang, Yixin Lan, Jie Xiao}
\address{School of Mathematical Sciences, Peking University, Beijing 100871, P. R. China}
\email{fangjp@math.pku.edu.cn (J. Fang)}
\address{Department of Mathematical Sciences, Tsinghua University, Beijing 100084, P. R. China}
\email{lanyx18@mails.tsinghua.edu.cn (Y. Lan)}
\address{Department of Mathematical Sciences, Tsinghua University, Beijing 100084, P. R. China}
\email{jxiao@tsinghua.edu.cn (J. Xiao)}

\date{\today}


\bibliographystyle{abbrv}

\maketitle

\begin{abstract}
Our investigation in the present paper is based on three important results. (1) In \cite{Ringel-1990}, Ringel introduced Hall algebra for representations of a quiver over finite fields and proved the elements corresponding to simple representations satisfy the quantum Serre relation. This gives a realization of the nilpotent part of quantum group if the quiver is of finite type. (2) In \cite{Green-1995}, Green found a homological formula for the representation category of the quiver and equipped Ringel's Hall algebra with a comultiplication. The generic form of the composition subalgebra of Hall algebra generated by simple representations realizes the nilpotent part of quantum group of any type. (3) In \cite{Lusztig-1991}, Lusztig defined induction and restriction functors for the perverse sheaves on the variety of representations of the quiver which occur in the direct images of constant sheaves on flag varieties, and he found a formula between his induction and restriction functors which gives the comultiplication as algebra homomorphism for quantum group. In the present paper, we prove the formula holds for all semisimple complexes with Weil structure. This establishes the categorification of Green's formula.
\end{abstract}

\setcounter{tocdepth}{1}\tableofcontents

\section{Introduction}

Let $U_v$ be a quantum group defined by Cartan datum induced by a quiver $Q=(I,H,s,t)$ as in \cite{Lusztig-1993}, and $U_v^-$ be its negative part. As a bialgebra, $U_v^-$ has a multiplication $m:U_v^-\otimes U_v^-\rightarrow U_v^-$ and a comultiplication $\Delta:U_v^-\rightarrow U_v^-\otimes U_v^-$ such that $\Delta$ is an algebra homomorphism. Realization of quantum group is an important problem in representation theory.

In \cite{Ringel-1990} and \cite{Ringel-1993}, Ringel realized $U_v^-$ by (twisted) Hall algebras. Let $\bbF_q$ be the finite field of order $q$ and $\Lambda$ be the set of isomorphism classes of finite-dimensional representations of $Q$ over $\bbF_q$. For any $\alpha\in \Lambda$, let $u_\alpha$ be a symbol and $M_\alpha$ be a fixed representation whose isomorphism class is $\alpha$. The twisted Hall algebra $H_q(Q)$ is a $\bbQ(v_q)$-vector space with the basis $\{u_\alpha\mid \alpha\in \Lambda\}$, where $v_q\in \bbC$ is a fixed square root of $q$. It has a multiplication 
$$u_\alpha*u_\beta=\sum_{\gamma\in \Lambda}v_q^{\langle\alpha,\beta\rangle}g^\gamma_{\alpha\beta}u_\gamma,$$
where $\langle-,-\rangle$ is the Euler form, and the filtration number $g^\gamma_{\alpha\beta}$ is the number of submodules $B$ of $M_\gamma$ such that $M_\gamma/B\simeq M_\alpha, B\simeq M_\beta$. Ringel proved that the elements corresponding to simple representations satisfy the quantum Serre relation and $H_q(Q)$ is isomorphic to $U_{v=v_q}^-$ as algebras if $Q$ is of finite type. It is natural to consider how to define a comultiplication for the Hall algebra $H_q(Q)$.

In \cite{Lusztig-1990} and \cite{Lusztig-1991}, Lusztig categorified $U_v^-$ by perverse sheaves. Let $k=\overline{\bbF}_q$ be the algebraic closure of $\bbF_q$. For any $\nu\in \bbN I$, fix a $I$-graded $k$-vector space $V_\nu$ with a $\bbF_q$-structure of dimension vector $\nu$, then the $k$-variety $\rE_{V_\nu}=\bigoplus_{h\in H}\Hom_{k}((V_\nu)_{h'},(V_{\nu})_{h''})$ together an action by the algebraic group $\rG_{V_\nu}=\prod_{i\in I}GL_{k}((V_\nu)_i)$ parametrizes the isomorphism classes of representations of $Q$ of dimension vector $\nu$. Lusztig defined $\cQ_{V_{\nu}}$ as the subcategory of $\cD^b_{\rG_{V_\nu}}(\rE_{V_\nu})$ consisting of direct sums of perverse sheaves (up to shifts and Tate twists) which occur as direct summands in the direct images of constant sheaves on flag varieties, see subsection \ref{categorification} for details. For any $\nu=\nu'+\nu''\in \bbN I$, Lusztig defined two functors
\begin{align*}
\Ind^\nu_{\nu',\nu''}&:\cD^b_{\rG_{V_{\nu'}}}(\rE_{V_{\nu'}})\times  \cD^b_{\rG_{V_{\nu''}}}(\rE_{V_{\nu''}})\rightarrow \cD^b_{\rG_{V_{\nu}}}(\rE_{V_{\nu}}),\\
\Res^\nu_{\nu',\nu''}&:\cD^b_{\rG_{V_{\nu}}}(\rE_{V_{\nu}})\rightarrow\cD^b_{\rG_{V_{\nu'}}\times \rG_{V_{\nu''}}}(\rE_{V_{\nu'}}\times\rE_{V_{\nu''}})
\end{align*}
satisfying $\Ind^\nu_{\nu',\nu''}(\cQ_{V_{\nu'}}\boxtimes \cQ_{V_{\nu''}})\subset \cQ_{V_{\nu}}, \Res^\nu_{\nu',\nu''}(\cQ_{V_{\nu}})\subset \cQ_{V_{\nu'}}\boxtimes \cQ_{V_{\nu''}}$, 
which are called induction and restriction functors, such that all induction and restriction functors induce a multiplication and a comultiplication on $\cK=\bigoplus_{\nu\in \bbN I}\cK_{\nu}$, where $\cK_\nu$ be the Grothendieck group of $\cQ_{V_{\nu}}$ which has a $\bbZ[v,v^{-1}]$-module structure via $v.[L]=[L[1](\frac{1}{2})]$. Moreover, Lusztig proved a formula about induction and restriction functors
$$\Res^\gamma_{\alpha',\beta'} \Ind^\gamma_{\alpha,\beta}(A\boxtimes B)\simeq \bigoplus_{\lambda\in \cN}\Ind (A_\lambda\boxtimes B_\lambda)[-2g(\lambda)](-g(\lambda))$$
for $A\in \cQ_{V_\alpha}, B\in \cQ_{V_\beta}$, see Proposition 8.4 in \cite{Lusztig-1991} or Lemma 13.1.5 in \cite{Lusztig-1993} for details, such that $\cK$ is a bialgebra, that is, the comultiplication $\cK\rightarrow \cK\otimes_{\bbZ[v,v^{-1}]}\cK$ is an algebra homomorphism with respect to a twisted multiplication on $\cK\otimes_{\bbZ[v,v^{-1}]}\cK$, and proved that the $\bbQ(v)$-algebra $\bbQ(v)\otimes_{\bbZ[v,v^{-1}]}\cK$ is isomorphic to $U_v^-$, see Theorem 13.2.11 in \cite{Lusztig-1993}.

In \cite{Green-1995}, Green equipped the Hall algebra $H_q(Q)$ with a comultiplication
$$\Delta(u_\gamma)=\sum_{\alpha,\beta\in \Lambda}v_q^{\langle\alpha,\beta\rangle}a_{\alpha}a_{\beta}a_\gamma^{-1}g^\gamma_{\alpha\beta}u_\alpha\otimes u_\beta,$$
where $a_\gamma$ is the order of the automorphism group of $M_\gamma$ for any $\gamma\in \Lambda$. He proved a homological formula about filtration numbers
\begin{align*}
&a_\alpha a_\beta a_{\alpha'}a_{\beta'}\sum_{\gamma\in \Lambda}a_\gamma^{-1}g^\gamma_{\alpha\beta}g^\gamma_{\alpha'\beta'}\\=&\sum_{\alpha_1,\alpha_2,\beta_1,\beta_2\in \Lambda}\frac{|\Ext_{kQ}^1(M_{\alpha_1},M_{\beta_2})|}{|\Hom_{kQ}(M_{\alpha_1},M_{\beta_2})|}g^\alpha_{\alpha_1\alpha_2}g^\beta_{\beta_1,\beta_2}g^{\alpha'}_{\alpha_1\beta_1}g^{\beta'}_{\alpha_2\beta_2}a_{\alpha_1}a_{\alpha_2}a_{\beta_1}a_{\beta_2},
\end{align*}
which is called Green's formula. This formula is equivalent to $\Delta$ being an algebra homomorphism for $H_q(Q)$ with respect to a twisted multiplication on $H_q(Q)\otimes H_q(Q)$. Green also generalized Ringel's result to the case $Q$ is of any type, and the generic form of the composition subalgebra of $H_q(Q)$ generated by simple representations is isomorphic to $U_v^-$ as bialgebras.

The link between Hall algebras realization and perverse sheaves realization is given by sheaf-function correspondence, see \cite{Lusztig-1998}, \cite{Xiao-Xu-Zhao-2019} and \cite{Kiehl-Rainer-2001}. On the one hand, the Hall algebra $H_q(Q)$ can be rewritten via functions. Let $F(x)=x^q$ be the Frobenius automorphism of $k$ and $\rE_{V_\nu}^F, \rG_{V_\nu}^F$ be the $F$-fixed subvarieties of $\rE_{V_\nu}, \rG_{V_\nu}$  for any $\nu\in \bbN I$. Let $\tilde{H}_\nu$ be the space of ${\rG_{V_\nu}^F}$-invariant functions on $\rE_{V_\nu}^F$. There are morphisms 
$$\ind^\nu_{\nu',\nu''}:\tilde{H}_{\nu'}\otimes \tilde{H}_{\nu''}\rightarrow \tilde{H}_\nu,\ \res^\nu_{\nu',\nu''}:\tilde{H}_\nu\rightarrow \tilde{H}_{\nu'}\otimes \tilde{H}_{\nu''}$$
which are defined in an analogue way as induction and restriction functors. Then $\tilde{H}_q(Q)=\bigoplus_{\nu\in \bbN I}\tilde{H}_v$ is isomorphic to the Hall algebra $H_q(Q)$ and $\ind,\res$ coincide with $*,\Delta$ respectively, see \cite{Xiao-Xu-Zhao-2019}. On the other hand, Lusztig proved that each complex in $\cQ_{V_\nu}$ has a canonical Weil structure in \cite{Lusztig-1998}. By Grothendieck's trace formula, the trace map $\chi_q:\cK\rightarrow \tilde{H}_q(Q)$ is a bialgebra homomorphism, that is, $\ind^\nu_{\nu',\nu''}(\chi_q\otimes \chi_q)=\chi_q\Ind^\nu_{\nu',\nu''}, \res^\nu_{\nu',\nu''}\chi_q=(\chi_q\otimes\chi_q)\Res^\nu_{\nu',\nu''}$. 

The image $\chi_q(\cK)$ coincide with the composition subalgebra of $\tilde{H}_q(Q)\simeq H_q(Q)$ (which are isomorphic to $U_{v=v_q}^-$). It is interesting to consider all semisimple complexes with Weil structure on $\rE_{V_{\nu}}$ beyond $\cQ_{V_\nu}$, because they have potential to categorify the Hall algebra $H_q(Q)$, see \cite{Xiao-Xu-Zhao-2019}. Under this sheaf-function correspondence, it is natural to consider which formula about semisimple complexes will correspond to the Green's formula. It is not difficult to guess the following formula, see Proposition 8.4 in \cite[Proposition 8.4]{Lusztig-1991}, Remark in Section 4.2 of \cite{Schiffmann-2006} and Theorem 7 in \cite{Xiao-Xu-Zhao-2019}. The notations will be explained in section 3.
\begin{theorem}
For any $A\in \cD^{b,ss}_{\rG_{V_\alpha},m}(\rE_{V_\alpha}), B\in \cD^{b,ss}_{\rG_{V_\beta},m}(\rE_{V_\beta})$, we have
\begin{align*}
&\Res^{\gamma}_{\alpha',\beta'}\Ind^{\gamma}_{\alpha,\beta}(A\boxtimes B)\simeq \\&\bigoplus_{\lambda=(\alpha_1,\alpha_2,\beta_1,\beta_2)\in \cN}\!\!\!\!\!\!\!(\Ind^{\alpha'}_{\alpha_1,\beta_1}\!\boxtimes\! \Ind^{\beta'}_{\alpha_2,\beta_2})(\tau_\lambda)_!((\Res^{\alpha}_{\alpha_1,\alpha_2}A)\!\boxtimes\! (\Res^{\beta}_{\beta_1,\beta_2}B))[-(\alpha_2,\beta_1)](-\frac{(\alpha_2,\beta_1)}{2}).
\end{align*}
\end{theorem}

For $A\in\cQ_{V_\alpha}, B\in \cQ_{V_\beta}$, it is equivalent to the formula given by Lusztig, see \cite[Proposition 8.4]{Lusztig-1991} or \cite[Lemma 13.1.5]{Lusztig-1993}.

Beyond $\cQ_{V_\alpha}, \cQ_{V_\beta}$, Xiao-Xu-Zhao verified the formula by trace maps and Green's formula, see Theorem 7 in \cite{Xiao-Xu-Zhao-2019}. More precisely, their strategy was based on the fact that two semisimple complexes are isomorphic if and only if they have the same images under trace maps $\chi_{q^n}$ over each finite field $\bbF_{q^n}$, see \uppercase\expandafter{\romannumeral3}.Theorem 12.1 in \cite{Kiehl-Rainer-2001}. By applying the trace map $\chi_{q^n}$, they obtained a formula about functions which is equivalent to the Green's formula over $\bbF_{q^n}$.

Consider these two facts: the comultiplication $\cK\rightarrow \cK\otimes_{\bbZ[v,v^{-1}]}\cK$ induced by Lusztig's restriction functors is an algebra homomorphism, and the comultiplication $H_q(Q)\rightarrow H_q(Q)\otimes H_q(Q)$ given by Green is an algebra homomorphism (both with respect to a twisted multiplication on the tensor products), where the former follows from the formula in Theorem 1.1 for $\cQ_{V_\alpha},\cQ_{V_{\beta}}$ and the latter follows from Green's formula. In the present paper, we give a sheaf-level proof of the formula in Theorem 1.1 beyond $\cK$ without using of trace maps, which is logically independent of Green's formula. This formula completes the categorification of Green's formula.

In section 2, we review some properties of mixed equivariant semisimple complexes, the definitions of induction, restriction functors and a result of hyperbolic localization functors. In section 3, we prove the formula by three subsections. In the first subsection, we deal with the left hand side of the formula with some methods inspired by Section 4.3-4.7 in \cite{Lusztig-1991}. In the second subsection, we deal with the right hand side of the formula. In the third subsection, we combine two sides with some methods inspired by \cite{Green-1995} and \cite{Ringel-1996}.

\section{Operations on perverse sheaves from quivers, a brief review}

Let $q,l$ be two fixed distinct prime numbers. We denote by $k=\overline{\bbF}_q$ the algebraic closure of the finite field of order $q$ and $\overline{\bbQ}_l$ the algebraic closure of the field of $l$-adic numbers. We fix an isomorphism $\overline{\bbQ}_l\simeq \bbC$.

\subsection{Mixed equivariant semisimple complex}\

In this subsection, we review some properties of mixed equivariant semisimple complexes. We refer \cite{Bernstein-Lunts-1994,Beilinson-Bernstein-Deligne-1982,Kiehl-Rainer-2001,Lusztig-1993} for details.

Let $X$ be a $k$-variety admitting an $\bbF_q$-structure, we denote by $\cD^b(X)$ the bounded derived category of constructible $\overline{\bbQ}_l$-sheaves on $X$, $\cD^b_m(X)$ the subcategory consisting of mixed complexes and $\cD^{b,ss}_m(X)$ the subcategory consisting of mixed semisimple complexes.

For any $n\in \bbZ$, we denote by $[n]$ the shift functor, $(\frac{n}{2})$ the Tate twist if $n$ is even or the square root of the Tate twist if $n$ is odd, and $\lp H^n$ the perverse cohomology functor.

Let $\cD^b_{\leqslant \omega}(X)$ and $\cD^b_{\geqslant \omega}(X)$ be the full subcategories of $\cD^b_m(X)$ consisting of complexes whose $i$-th cohomology has weight smaller than $\omega+i$ and larger than $\omega+i$ for any $i$ respectively. Any objects in $\cD^b_{\leqslant \omega}\cap \cD^b_{\geqslant \omega}$ are said to be pure of weight $\omega$. Note that simple perverse sheaves are pure, since each perverse sheaf has a canonical filtration with pure subquotients, see Theorem 5.3.5 in \cite{Beilinson-Bernstein-Deligne-1982} or Theorem 5.4.12 in \cite{Pramod-2021}.

If $f:X\rightarrow Y$ is a morphism between $k$-varieties, the derived functors of $f^*,f_*,f_!$ are still denoted by $f^*:\cD^b(Y)\rightarrow \cD^b(X), f_*,f_!:\cD^b(X)\rightarrow \cD^b(Y)$ respectively. The functor $f_!$ has a right adjoint $f^!:\cD^b(Y)\rightarrow \cD^b(X)$.

\begin{proposition}[\cite{Beilinson-Bernstein-Deligne-1982}, Section 4.2.4, 5.1.14]\label{pure-*!}
These functors $f^*,f_*,f^!,f_!,[n],(\frac{n}{2})$ send mixed complexes to mixed complexes. Moreover, \\
(a) if $K$ is pure of weight $\omega$, then $K[n]$ is pure of weight $\omega+n$ and $K(\frac{n}{2})$ is pure of weight $\omega-n$;\\
(b) $f^*,f_!$ preserve $\cD^b_{\leqslant \omega}$;\\
(c) $f_*,f^!$ preserve $\cD^b_{\geqslant \omega}$;\\
(d) if $f$ is smooth with connected fibres of dimension $d$, then $f^!=f^*[2d](d)$, and so $f^*$ sends pure complexes of weight $\omega$ to pure complexes of weight $\omega$.\\
(e) if $f$ is proper, then $f_!=f_*$, and so $f_!$ sends pure complexes of weight $\omega$ to pure complexes of weight $\omega$.
\end{proposition}

\begin{theorem}[\cite{Beilinson-Bernstein-Deligne-1982}, Theorem 5.3.8]\label{BBD}
If $K\in \cD^b_m(X)$ is pure, then it is semisimple.
\end{theorem}

\begin{corollary}\label{semisimple-*!} Let $f:X\rightarrow Y$ be a morphism between $k$-varieties, then \\
(a) if $f$ is smooth with connected fibres, then $f^*$ sends semisimple complexes on $Y$ to semisimple complexes on $X$;\\
(b) if $f:X\rightarrow Y$ is proper, then $f_!$ sends semisimple complexes on $X$ to semisimple complexes on $Y$.
\end{corollary}
\begin{remark}
The statement (b) is known as Beilinson-Bernstein-Deligne decomposition theorem.
\end{remark}

Let $X$ be a $k$-variety admitting an $\bbF_q$-structure, and $G$ be a connected algebraic group admitting an $\bbF_q$-structure such that $G$ acts on $X$. We denote by $\cD_{G}^b(X)$ the $G$-equivariant bounded derived category of constructible $\overline{\bbQ}_l$-sheaves on $X$, $\cD_{G,m}^b(X)$ the subcategory consisting of mixed complexes and $\cD_{G,m}^{b,ss}(X)$ the subcategory consisting of mixed semisimple complexes.

\begin{proposition}[\cite{Lusztig-1993}, Section 8.1.7, 8.1.8]\label{Lusztig-principal}
Assume that $G$ acts on two $k$-varieties $X,Y$, and $f:X\rightarrow Y$ is a $G$-equivariant morphism, then\\
(a) $f^*,f_!$ send $G$-equivariant complexes to $G$-equivariant complexes;\\
(b) if $f$ is a locally trivial principal $G$-bundle, then there exists a functor $f_\flat$ which is a quasi-inverse of $f^*$ defining an equivalence of categories 
\begin{diagram}[midshaft]
\cD^{b,ss}_{G,m}(X) &\pile{\rTo^{f_\flat}\\ \lTo_{f^*}} &\cD^{b,ss}_m(Y),
\end{diagram}
where $f_\flat(K)=\bigoplus_{n\in \bbZ}\lp H^{n-\dim G}(f_*(\lp H^nK[-n]))[-n+\dim G]$ for $K\in \cD^{b,ss}_{G,m}(X)$.
\end{proposition}

\begin{corollary}\label{principal}
If $f:X\rightarrow Y$ is a locally trivial principal $G$-bundle, then $f_\flat:\cD^{b,ss}_{G,m}(X)\rightarrow \cD^{b,ss}_m(Y)$ sends pure complexes of weight $\omega$ to pure complexes of weight $\omega$.
\end{corollary}
\begin{proof}
For any pure complex $K\in \cD^{b,ss}_{G,m}(X)$ of weight $\omega$, since $f_*$ preserves $\cD^b_{\geqslant \omega}$, we have $f_\flat(K)\in \cD^b_{\geqslant \omega}(Y)$. Suppose $\omega'=\inf\{\omega''\mid f_\flat(K)\in \cD^b_{\leqslant \omega''}(Y)\}$, then $K\simeq f^*f_\flat(K)\in \cD^b_{\leqslant w'}(X)$, since $f^*$ preserves $\cD^b_{\leqslant \omega'}$. By $K\in \cD^b_{\leqslant \omega}(X)\cap\cD^b_{\geqslant \omega}(X)$, we have $\omega'=\omega$, and so $f_\flat(K)\in \cD^b_{\leqslant \omega}(Y)\cap\cD^b_{\geqslant \omega}(Y)$ is pure of weight $\omega$.
\end{proof}

\begin{proposition}[\cite{Pramod-2021}, Theorem 6.6.16]\label{sub-quotient-equivariant}
Assume that $G$ acts on a $k$-variety $X$, then\\
(a) if $H$ is a subgroup of $G$, then there is a forgetful functor $\cD^b_{G,m}(X)\xrightarrow{\textrm{forget}}\cD^b_{H,m}(X)$, which is not fully faithful in general.\\
(b) if $U$ is a unipotent normal subgroup of $G$ which acts trivially on $X$, suppose $G=T\ltimes U$ such that $G/U\simeq T$, then the forget functor $\cD^b_{G,m}(X)\xrightarrow{\textrm{forget}}\cD^b_{T,m}(X)$ is an equivalence.
\end{proposition}

\subsection{Hyperbolic localization}\ 

In this subsection, we review a result in \cite{Braden-2003}. Another proof without the assumption that $X$ is a normal variety is given by \cite{Drinfeld-Gaitsgory-2014}.

Let $X$ be a (normal) $k$-variety together with a $k^*$-action. We denote by $X^{k^*}$ the subvariety of fixed points with connected components $X_1,...,X_r$, and define
\begin{align*}
X_i^+=\{x\in X\mid \lim_{t\rightarrow 0}t.x\in X_i\}
\end{align*}
for $i=1,...,r$. Let $X^+=\bigsqcup^r_{i=1}X_i^+$ be the disjoint union of them, and 
\begin{align*}
&f^+:X^{k^*}=\bigsqcup^r_{i=1}X_i\hookrightarrow \bigsqcup^r_{i=1}X_i^+=X^+,\\
&g^+:X^+\hookrightarrow X
\end{align*}
be inclusions. Then the hyperbolic localization functor is defined by 
\begin{align*}
(-)^{!*}:\cD^b(X)&\rightarrow \cD^b(X^{k^*})\\
K&\mapsto(f^+)^!(g^+)^*(K)
\end{align*}

An object $K$ in $\cD^b(X)$ is said to be weakly equivariant, if $\mu^*(K)\simeq L\boxtimes K$ for some locally constant sheaf $L$ on $k^*$, where $\mu:k^*\times X\rightarrow X$ is the map defining the action. Note that if $K$ is a $k^*$-equivariant perverse sheaf (or more generally, a $k^*$-equivariant semisimple complex) on $X$, then $K$ is weakly equivaraint, since $\mu^*(K)\simeq p^*(K)\simeq \overline{\bbQ}_l\boxtimes K$, where $p:k^*\times X\rightarrow X$ is the natural projection.
 
\begin{theorem}[\cite{Braden-2003},Theorem 8]\label{hyperbolic}
The hyperbolic localization functor preserves purity of weakly equivariant mixed complexes, and so it sends semisimple complexes to semisimple complexes.
\end{theorem}

\begin{remark}
By the proof of this theorem, the hyperbolic localization functors preserve the weights of pure weakly equivariant complexes.
\end{remark}

The hyperbolic localization functors have another descriptions. We define 
\begin{align*}
\pi^+:X^+&\rightarrow X^{k^*}\\
x&\mapsto \lim_{t\rightarrow 0}t.x
\end{align*}
then for any weakly equivariant object $K$ in $\cD^b(X)$, by the formula (1) in Section 3 of \cite{Braden-2003}, there are natural isomorphisms
\begin{align*} 
K^{!*}\simeq (\pi^+)_!(g^+)^*(K).
\end{align*}

\subsection{Moduli variety of representations}\

In this subsection, we set up the $k$-variety we work on.

Let $Q=(I,H,s,t)$ be a finite quiver, where $I$ is the set of vertices, $H$ is the set of arrows, $s(h)\in I$ is the source and $t(h)\in I$ is the target for an arrow $h\in H$. We denote by $s(h)=h',t(h)=h''$ for simplicity. 

For each $\nu\in \bbN I$, we fix a $I$-graded $k$-vector space $V_\nu$ of dimension vector $\nu$ with a given $\bbF_q$-rational structure with Frobenius map $F:V_\nu\rightarrow V_\nu$, see \cite{Lusztig-1998}, and define the affine space 
$$\rE_{V_\nu}=\bigoplus_{h\in H}\Hom_k((V_\nu)_{h'},(V_\nu)_{h''}).$$
The algebraic group 
$$\rG_{V_\nu}=\prod_{i\in I}GL_k((V_\nu)_i)$$ 
acts on $\rE_{V_\nu}$ by $(g.x)_h=g_{h''}x_hg_{h'}^{-1}$ for $g\in \rG_{V_\nu},x\in \rE_{V_\nu},h\in H$.

The Euler form and the symmetric Euler form are defined by 
\begin{align*}
\langle\nu,\nu'\rangle&=\sum_{i\in I}\nu_i\nu'_i-\sum_{h\in H}\nu_{h'}\nu'_{h''},\\
(\nu,\nu')&=\langle\nu,\nu'\rangle+\langle\nu',\nu\rangle
\end{align*}
for any $\nu,\nu'\in \bbN I$ respectively.

\subsection{Induction functor}\

In this subsection, we review the definition of induction functor, see \cite[Section 3]{Lusztig-1991} or \cite[Section 9.2]{Lusztig-1993} for more details.

Given $\nu\in \bbN I$ and $x\in \rE_{V_\nu}$, a $I$-graded subspace $W$ of $V_\nu$ is said to be $x$-stable, if $x_h(W_{h'})\subset W_{h''}$ for any $h\in H$. 

Given $\nu=\nu'+\nu''\in \bbN I$, $x\in \rE_{V_\nu}$ and a $x$-stable $I$-graded subspace $W$ of dimension vector $\nu''$, we denote by $x|_W:W\rightarrow W$ the restriction of $x$ to $W$, and denote by $\overline{x}^W:V_\nu/W\rightarrow V_\nu/W$ the quotient of $x$ on $V_\nu/W$. Given $I$-graded linear isomorphisms $\rho_1:V_\nu/W\xrightarrow{\simeq} V_{\nu'}$ and $\rho_2:W\xrightarrow{\simeq}V_{\nu''}$, we define $\rho_{1*}\overline{x}^W\in \rE_{V_{\nu'}}$ and $\rho_{2*}x|_W\in \rE_{V_{\nu''}}$ by $(\rho_{1*}\overline{x}^W)_h=(\rho_{1h''})(\overline{x}^W)_h(\rho_{1h'})^{-1}$ and $(\rho_{2*}x|_W)_h=(\rho_{2h''})(x|_W)_h(\rho_{2h'})^{-1}$ for any $h\in H$ respectively.

Let $E''^{\nu}_{\nu',\nu''}$ be the variety of pairs $(x,W)$, where $x\in \rE_{V_\nu}$ and $W$ is a $x$-stable $I$-graded subspace of $V_\nu$ of dimension vector $\nu''$. Let $E'^{\nu}_{\nu',\nu''}$ be the variety of quadruples $(x,W,\rho_1,\rho_2)$, where $(x,W)\in E''^{\nu}_{\nu',\nu''}$ and $\rho_1:V_\nu/W\xrightarrow{\simeq} V_{\nu'},\rho_2:W\xrightarrow{\simeq}V_{\nu''}$ are $I$-graded linear isomorphisms. Then $\rG_{V_{\nu'}}\times \rG_{V_{\nu''}}\times \rG_{V_\nu}$ acts on $E'^{\nu}_{\nu',\nu''}$ by $(g_1,g_2,g).(x,W,\rho_1,\rho_2)=(g.x,g(W),g_1\rho_1g^{-1},g_2\rho_2g^{-1})$ for $(g_1,g_2,g)\in \rG_{V_{\nu'}}\times \rG_{V_{\nu''}}\times \rG_{V_\nu}, (x,W,\rho_1,\rho_2)\in E'^{\nu}_{\nu',\nu''}$, and $\rG_{V_\nu}$ acts on $E''^{\nu}_{\nu',\nu''}$ by $g.(x,W)=(g.x, g(W))$ for $g\in \rG_{V_\nu},(x,W)\in E'^{\nu}_{\nu',\nu''}$.

Consider the following diagram
\begin{diagram}[midshaft,size=2em]
\rE_{V_{\nu'}}\times \rE_{V_{\nu''}} &\lTo^{{p_1}^{\nu}_{\nu',\nu''}} &E'^{\nu}_{\nu',\nu''} &\rTo^{{p_2}^{\nu}_{\nu',\nu''}} &E''^{\nu}_{\nu',\nu''} &\rTo^{{p_3}^{\nu}_{\nu',\nu''}} &\rE_{V_\nu}\\
(\rho_{1*}\overline{x}^W,\rho_{2*}x|_W)&\lMapsto &(x,W,\rho_1,\rho_2) &\rMapsto &(x,W) &\rMapsto &x,
\end{diagram}
where ${p_1}^{\nu}_{\nu',\nu''}$ is smooth with connected fibres and $\rG_{V_{\nu'}}\times \rG_{V_{\nu''}}\times \rG_{V_\nu}$-equivariant with respect to the trivial action of $\rG_{V_\nu}$ on $\rE_{V_{\nu'}}\times \rE_{V_{\nu''}}$, ${p_2}^{\nu}_{\nu',\nu''}$ is a principal $\rG_{V_{\nu'}}\times \rG_{V_{\nu''}}$-bundle and $\rG_{V_{\nu}}$-equivariant, ${p_3}^{\nu}_{\nu',\nu''}$ is proper and $\rG_{V_\nu}$-equivariant. By Corollary \ref{semisimple-*!} and Proposition \ref{Lusztig-principal}, we may define the induction functor by 
\begin{align*}
&\cD^{b,ss}_{\rG_{V_{\nu'},m}}(\rE_{V_{\nu'}})\boxtimes \cD^{b,ss}_{\rG_{V_{\nu''},m}}(\rE_{V_{\nu''}})\xrightarrow{-\boxtimes-}\cD^{b,ss}_{\rG_{V_{\nu'}}\times\rG_{V_{\nu''},m}}(\rE_{V_{\nu'}}\times \rE_{V_{\nu''}})\\
&=\cD^{b,ss}_{\rG_{V_{\nu'}}\times\rG_{V_{\nu''}}\times \rG_{V_\nu},m}(\rE_{V_{\nu'}}\times \rE_{V_{\nu''}})\xrightarrow{({p_1}^{\nu}_{\nu',\nu''})^*}\cD^{b,ss}_{\rG_{V_{\nu'}}\times \rG_{V_{\nu''}}\times \rG_{V_\nu},m}(E'^{\nu}_{\nu',\nu''})\\&\xrightarrow{({p_2}^{\nu}_{\nu',\nu''})_\flat}\cD^{b,ss}_{\rG_{V_\nu},m}(E''^{\nu}_{\nu',\nu''})\xrightarrow{({p_3}^{\nu}_{\nu',\nu''})_!}\cD^{b,ss}_{\rG_{V_\nu},m}(\rE_{V_\nu})\\
&\ \ \ \ \ \ \ \ \ \ \ \ \Ind^{\nu}_{\nu',\nu''}(A\boxtimes B)=({p_3}^{\nu}_{\nu',\nu''})_!({p_2}^{\nu}_{\nu',\nu''})_\flat({p_1}^{\nu}_{\nu',\nu''})^*(A\boxtimes B)[d_1-d_2](\frac{d_1-d_2}{2})\end{align*}
for $A\in \cD^{b,ss}_{\rG_{V_{\nu'}},m}(\rE_{V_{\nu'}}),B\in \cD^{b,ss}_{\rG_{V_{\nu''}},m}(\rE_{V_{\nu''}})$, where $d_1,d_2$ are the dimensions of the fibres of ${p_1}^{\nu}_{\nu',\nu''},{p_2}^{\nu}_{\nu',\nu''}$ respectively, and we have 
$$d_1-d_2=\sum_{h\in H}\nu'_{h'}\nu''_{h''}+\sum_{i\in I}\nu'_i\nu''_i.$$

\subsection{Restriction functor}\

In this subsection, we review the definition of restriction functor, see \cite[Section 4]{Lusztig-1991} or \cite[Section 9.2]{Lusztig-1993} for more details.

Given $\nu=\nu'+\nu''\in \bbN I$, we fix a $I$-graded subspace $W^{\nu''}$ of $V_\nu$ of dimension vector $\nu''$, and fix two $I$-graded linear isomorphisms $\rho_1^{\nu''}:V_\nu/W^{\nu''}\xrightarrow {\simeq}V_{\nu'}, \rho_2^{\nu''}:W^{\nu''}\xrightarrow{\simeq}V_{\nu''}$. 

Let $Q^{\nu''}\subset\rG_{V_\nu}$ be the stabilizer of $W^{\nu''}\subset V_{\nu}$ which is a parabolic subgroup, and let $U^{\nu''}\subset Q^{\nu''}$ be its unipotent radical, then there is a canonical isomorphism $Q^{\nu''}/U^{\nu''}\simeq \rG_{V_\nu'}\times \rG_{V_\nu''}$.

Let $F^{\nu}_{\nu',\nu''}$ be the closed subvariety of $\rE_{V_\nu}$ consisting of $x$ such that $W^{\nu''}$ is $x$-stable. Then $Q^{\nu''}$ acts on $F^{\nu}_{\nu',\nu''}$, and $Q^{\nu''}$ acts on $\rE_{V_{\nu'}}\times \rE_{V_{\nu''}}$ through the quotient $Q^{\nu''}/U^{\nu''}\simeq \rG_{V_{\nu'}}\times \rG_{V_{\nu''}}$. 

Consider the following diagram
\begin{diagram}[midshaft,size=2em]
\rE_{V_{\nu'}}\times \rE_{V_{\nu''}} &\lTo^{\kappa^{\nu}_{\nu',\nu''}} &F^{\nu}_{\nu',\nu''} &\rInto^{\iota^{\nu}_{\nu',\nu''}} &\rE_{V_\nu}\\
(\rho_{1*}^{\nu''}\overline{x}^{W^{\nu''}},\rho_{2*}^{\nu''}x|_{W^{\nu''}}) &\lMapsto &x &\rMapsto &x,
\end{diagram}
where $\kappa^{\nu}_{\nu',\nu''}$ is a vector bundle of rank $\sum_{h\in H}\nu'_{h'}\nu''_{h''}$ and $Q^{\nu''}$-equivariant, $\iota^{\nu}_{\nu',\nu''}$ is the inclusion and $Q^{\nu''}$-equivariant. 
Note that the group $U^{\nu''}$ acts trivially on $\rE_{V_{\nu'}}\times \rE_{V_{\nu''}}$, by Proposition 2.7, there is an equivalence $\cD^b_{Q^{\nu''},m}(\rE_{V_{\nu'}}\times \rE_{V_{\nu''}})\xrightarrow{\simeq}\cD^b_{\rG_{V_{\nu'}}\times \rG_{V_{\nu''}},m}(\rE_{V_{\nu'}}\times \rE_{V_{\nu''}})$, By Proposition \ref{Lusztig-principal} and Proposition \ref{sub-quotient-equivariant}, we may define the restriction functor by
\begin{align*}
&\cD^b_{\rG_{V_\nu},m}(\rE_{V_{\nu}})\xrightarrow{\textrm{forget}}\cD^b_{Q^{\nu''},m}(\rE_{V_{\nu}})\xrightarrow{(\iota^{\nu}_{\nu',\nu''})^*}\cD^b_{Q^{\nu''},m}(F^{\nu}_{\nu',\nu''})\\
&\xrightarrow{(\kappa^{\nu}_{\nu',\nu''})_!}\cD^b_{Q^{\nu''},m}(\rE_{V_{\nu'}}\times \rE_{V_{\nu''}})\xrightarrow{\simeq}\cD^b_{\rG_{V_{\nu'}}\times \rG_{V_{\nu''}},m}(\rE_{V_{\nu'}}\times \rE_{V_{\nu''}})\\
&\Res^{\nu}_{\nu',\nu''}(C)=(\kappa^{\nu}_{\nu',\nu''})_!(\iota^{\nu}_{\nu',\nu''})^*(C)[-\langle\nu',\nu''\rangle](-\frac{\langle\nu',\nu''\rangle}{2})
\end{align*}
for $C\in \cD^b_{\rG_{V_\nu},m}(\rE_{V_{\nu}})$.

\begin{proposition}\label{res-hyperbolic}
For any $\nu=\nu'+\nu''\in \bbN I$, the restriction functor $\Res^{\nu}_{\nu',\nu''}$ sends semisimple complexes to semisimple complexes.
\end{proposition}
\begin{proof}
We claim that 
$$\cD^b_{m}(\rE_{V_{\nu}})\xrightarrow{(\iota^{\nu}_{\nu',\nu''})^*}\cD^b_{m}(F^{\nu}_{\nu',\nu''})
\xrightarrow{(\kappa^{\nu}_{\nu',\nu''})_!}\cD^b_{m}(\rE_{V_{\nu'}}\times \rE_{V_{\nu''}})$$
is a hyperbolic localization functor. 

We fix a $I$-graded direct sum decomposition $V_\nu=\tilde{W}\oplus W^{\nu''}$ and a $I$-graded linear isomorphism $\rho:\tilde{W}\xrightarrow{\simeq}V_\nu/W^{\nu''}$. Then there is a bijection  
\begin{align*}
\rE_{V_{\nu'}}\!\times\! \rE_{V_{\nu''}}&\!\times\! \bigoplus_{h\in H}\Hom_k((V_{\nu'})_{h'},(V_{\nu''})_{h''})\!\times\! \bigoplus_{h\in H}\Hom_k((V_{\nu''})_{h'},(V_{\nu'})_{h''})\simeq \rE_{V_\nu},\\
&(x',x'',y,w)\mapsto \begin{pmatrix}
 \rho_1^{\nu''}\rho &\\
  &\rho_2^{\nu''}\end{pmatrix}^{-1}\begin{pmatrix}
 x' &w\\
 y &x''\end{pmatrix}\begin{pmatrix}
 \rho_1^{\nu''}\rho &\\
  &\rho_2^{\nu''}\end{pmatrix}
\end{align*}
where we write elements of $\rE_{V_\nu}$ as block matrices with respect to the decomposition $V_\nu=\tilde{W}\oplus W^{\nu''}$.

There is a one-parameter subgroup embedding $\zeta:k^*\hookrightarrow Q^{\nu''}$ given by $t\mapsto \textrm{Id}_{V_{\nu'}}\bigoplus t\textrm{Id}_{V_{\nu''}}$ such that $k^*$ acts on $E_{V_\nu}$ via $t.x=\zeta(t).x$. More precisely, under above bijection, the $k^*$-action is given by $t.(x',x'',y,w)=(x',x'',ty,t^{-1}w)$. Then there is a commutative diagram
\begin{diagram}[midshaft,size=2em]
(\rE_{V_\nu})^{k^*} &\lTo^{\pi^+}  &(\rE_{V_{\nu}})^+ &\rInto^{g^+} &\rE_{V_\nu}\\
\dTo^\simeq & &\dTo^\simeq & &\vEq\\
\rE_{V_{\nu'}}\times \rE_{V_{\nu''}} &\lTo^{\kappa^{\nu}_{\nu',\nu''}} &F^{\nu}_{\nu',\nu''} &\rInto^{\iota^{\nu}_{\nu',\nu''}} &\rE_{V_\nu}.
\end{diagram}
Hence $(\kappa^{\nu}_{\nu',\nu''})_!(\iota^{\nu}_{\nu',\nu''})^*:\cD^b_{m}(\rE_{V_{\nu}})\rightarrow \cD^b_{m}(\rE_{V_{\nu'}}\times \rE_{V_{\nu''}})$ is a hyperbolic localization functor. Moreover, for any $C\in \cD^{b,ss}_{Q^{\nu''},m}(\rE_{V_{\nu}})$ regarded as an object in $\cD^{b,ss}_{k^*,m}(\rE_{V_{\nu}})$ via the embedding $\zeta:k^*\hookrightarrow Q^{\nu''}$ and the forgetful functor $\cD^{b,ss}_{Q^{\nu''},m}(\rE_{V_{\nu}})\xrightarrow{\textrm{forget}}\cD^{b,ss}_{k^*,m}(\rE_{V_{\nu}})$, it is weakly equivariant. By Theorem \ref{hyperbolic}, $(\kappa^{\nu}_{\nu',\nu''})_!(\iota^{\nu}_{\nu',\nu''})^*(C)$ is a semisimple complex on $\rE_{V_{\nu'}}\times \rE_{V_{\nu''}}$, and by Proposition \ref{Lusztig-principal}, $(\kappa^{\nu}_{\nu',\nu''})_!(\iota^{\nu}_{\nu',\nu''})^*(C)\in \cD^{b,ss}_{Q^{\nu''},m}(\rE_{V_{\nu'}}\times \rE_{V_{\nu''}})$ is $Q^{\nu''}$-equivariant. Therefore, 
$$\Res^{\nu}_{\nu',\nu''}:\cD^{b,ss}_{\rG_{V_\nu},m}(\rE_{V_{\nu}})\rightarrow\cD^{b,ss}_{\rG_{V_{\nu'}}\times \rG_{V_{\nu''}},m}(\rE_{V_{\nu'}}\times \rE_{V_{\nu''}}).$$
\end{proof}

\subsection{Categorification of $U_v^-$}\label{categorification}\

In this subsection, we review Lusztig's categorification for $U_v^-$, see \cite[Chapter 9-13]{Lusztig-1993} or \cite{Schiffmann-2012} for more details.

For any $\nu\in \bbN I$, we denote by $\mathcal{V}_\nu$ the set of sequences of the form $\bnu=(\nu^1,...,\nu^m)$, where each $\nu^l$ is of the form $ni$ for some $n>0$ and $i\in I$, satisfying $\nu=\sum^m_{l=1}\nu^l$.

For any $\bnu\in \mathcal{V}_\nu$, a flag of type $\bnu$ is a sequence of $I$-graded subspaces 
$$f=(V_\nu=V^0\supset V^1\supset...\supset V^m=0),$$
where the dimension vector of $V^{l-1}/V^l$ is $\nu^l$ for $l=1,...,m$. Moreover, for $x\in \rE_{V_\nu}$, such a flag is said to be $x$-stable, if $x_h(V^l_{h'})\subset V^l_{h''}$ for $l=1,...,m$ and $h\in H$.

For any $\nu\in \bbN I,\bnu\in \mathcal{V}_\nu$, we denote by $\tilde{\cF}_{\bnu}$ the variety of pairs $(x,f)$, where $x\in \rE_{V_\nu}$ and $f$ is a $x$-stable flag of type $\bnu$. The algebraic group $\rG_{V_\nu}$ acts on it via $g.(x,f)=(g.x,g.f)$, where 
$$g.f=(V_\nu=g(V^0)\supset g(V^1)\supset...\supset g(V^m)=0)$$
for $g\in \rG_{V_\nu},(x,f)\in \tilde{\cF}_{\bnu}$.

\begin{lemma}[\cite{Lusztig-1993}, Section 9.1.3]
The variety $\tilde{\cF}_{\bnu}$ is smooth, irreducible and the first projection $\pi_{\bnu}:\tilde{\cF}_{\bnu}\rightarrow \rE_{V_{\nu}}$ is proper and $\rG_{V_{\nu}}$-equivariant. 
\end{lemma}

By decomposition theorem, see Corollary \ref{semisimple-*!}, the complex $L_{\bnu}=(\pi_{\bnu})_!(\overline{\bbQ}_l)[d(\bnu)](\frac{d(\bnu)}{2})$ is a $\rG_{V_{\nu}}$-equivariant semisimple complex on $\rE_{V_{\nu}}$, where $d(\bnu)$ is the dimension of $\tilde{\cF}_{\bnu}$.

We denote by $\cP_{V_{\nu}}$ the subcategory of $\cD^{b,ss}_{\rG_{V_{\nu}},m}(\rE_{V_{\nu}})$ consisting of simple perverse sheaves $L$ such that $L[n]$ is a direct summand of $L_{\bnu}$ for some $n\in \bbZ$ and $\bnu\in \mathcal{V}_\nu$. We denote by $\cQ_{V_{\nu}}$ the subcategory of $\cD^{b,ss}_{\rG_{V_{\nu}},m}(\rE_{V_{\nu}})$ consisting of finite direct sum of complexes of the form $L[n]$ for various $L\in \cP_{V_{\nu}}$ and various $n\in \bbZ$. We denote by $\cK_{\nu}$ the Grothendieck group of $\cQ_{V_{\nu}}$, and define a $\bbZ[v,v^-1]$-module structure on $\cK=\bigoplus_{\nu\in\bbN I}\cK_{\nu}$ via $v.[L]=[L[1](\frac{1}{2})]$.

\begin{lemma}[\cite{Lusztig-1993}, Section 9.2.7, 9.2.11]\label{ind,res}
Given $\nu=\nu'+\nu''\in \bbN I$, for any $\bnu\in \mathcal{V}_\nu,\bnu'\in \mathcal{V}_{\nu'},\bnu''\in \mathcal{V}_{\nu''}$, we have 
\begin{align*}
&\Ind^{\nu}_{\nu',\nu'}(L_{\bnu'}\boxtimes L_{\bnu''})=L_{\bnu'\bnu''},\\
&\Res^{\nu}_{\nu',\nu'}(L_{\bnu})=\bigoplus L_{\boldsymbol{\tau}}\boxtimes L_{\boldsymbol{\omega}}[M'(\boldsymbol{\tau},\boldsymbol{\omega})](\frac{M'(\boldsymbol{\tau},\boldsymbol{\omega})}{2}),
\end{align*}
where the direct sum in the second formula is taken over $(\boldsymbol{\tau},\boldsymbol{\omega})\in \mathcal{V}_{\nu'}\times \mathcal{V}_{\nu''}$ satisfying $\nu^l=\tau^l+\omega^l$ for all $l$, and $M'(\boldsymbol{\tau},\boldsymbol{\omega})$ can be found in \cite[Section 9.2.11]{Lusztig-1993}. 
\end{lemma}

\begin{lemma}[\cite{Lusztig-1993}, Proposition 12.6.3]\label{generator}
As a $\bbZ[v,v^-1]$-module, $\cK$ is generated by $L_{\bnu}$ for all $\bnu\in \mathcal{V}_{\nu},\nu\in \bbN I$.
\end{lemma}

All induction functors induce a multiplication $m:\cK\otimes_{\bbZ[v,v^{-1}]}\cK\rightarrow \cK$. It is associative, since it is easy to check the formula $m(m\otimes \textrm{Id})(a\otimes b\otimes c)=m(\textrm{Id}\otimes m)(a\otimes b\otimes c)$ if $a,b,c$ are of the form $[L_{\bnu}]$ by Lemma \ref{ind,res}, then the formula holds for any $a,b,c\in \cK$ by Lemma \ref{generator}. Similarly, restriction functors induce a comultiplication $\Delta:\cK\rightarrow \cK\otimes_{\bbZ[v,v^{-1}]}\cK$ which is coassociative. Similarly, we have the following Lemma.

\begin{lemma}[\cite{Lusztig-1993}, Lemma 13.1.5]\label{bialgebra}
The comultiplication $\Delta:\cK\rightarrow \cK\otimes_{\bbZ[v,v^{-1}]}\cK$ is an algebra homomorphism, where the multiplication on $\cK\otimes_{\bbZ[v,v^{-1}]}\cK$ is given by $(a\otimes b)(a'\otimes b')=v^{-(\nu,\nu')}m(a\otimes a')\otimes m(b\otimes b')$ for $b\in \cK_{\nu},a'\in \cK_{\nu'}$.
\end{lemma}

\begin{remark}
(a) By Lemma \ref{bialgebra}, $\cK$ is a $\bbZ[v,v^{-1}]$-bialgebra. Moreover, the $\bbQ(v)$-algebra $\bbQ(v)\otimes_{\bbZ[v,v^{-1}]}\cK$ is isomorphic to $U_v^-$ as bialgebras, see \cite[Theorem 13.2.11]{Lusztig-1993}.\\
(b) Lemma \ref{bialgebra} is equivalent to the fact that Theorem \ref{main} holds for $A\in \cQ_{V_\alpha},B\in \cQ_{V_\beta}$.
\end{remark}

\section{Main theorem and proof}\

In this section, we prove the main theorem.

For any $\gamma=\alpha+\beta=\alpha'+\beta'\in \bbN I$, we let $\cN$ be the set of quadruples $\lambda=(\alpha_1,\alpha_2,\beta_1,\beta_2)\in (\bbN I)^4$ satisying $\alpha=\alpha_1+\alpha_2,\beta=\beta_1+\beta_2,\alpha'=\alpha_1+\beta_1,\beta'=\alpha_2+\beta_2$, and let 
\begin{align*}
\tau_\lambda:\rE_{V_{\alpha_1}}\times \rE_{V_{\alpha_2}}\times \rE_{V_{\beta_1}}\times \rE_{V_{\beta_2}}&\xrightarrow{\simeq}\rE_{V_{\alpha_1}}\times \rE_{V_{\beta_1}}\times \rE_{V_{\alpha_2}}\times \rE_{V_{\beta_2}}\\
(x_{\alpha_1},x_{\alpha_2},x_{\beta_1},x_{\beta_2})&\mapsto (x_{\alpha_1},x_{\beta_1},x_{\alpha_2},x_{\beta_2})
\end{align*}
be the isomorphism switching the second coordinate and the third coordinate.

In order to define functors $\Res^{\gamma}_{\alpha',\beta'},\Res^{\alpha}_{\alpha_1,\alpha_2}$ and $\Res^{\beta}_{\beta_1,\beta_2}$, we need to fix $I$-graded subspaces $W^{\beta'}\subset V_\gamma, W^{\alpha_2}\subset V_\alpha, W^{\beta_2}\subset V_\beta$ of dimension vector $\beta',\alpha_2,\beta_2$ respectively, and fix $I$-graded linear isomorphisms $\rho_1^{\beta'}:V_\gamma/W^{\beta'}\xrightarrow{\simeq}V_{\alpha'},\rho_2^{\beta'}:W^{\beta'}\xrightarrow{\simeq}V_{\beta'},\rho_1^{\alpha_2}:V_\alpha/W^{\alpha_2}\xrightarrow{\simeq}V_{\alpha_1},\rho_2^{\alpha_2}:W^{\alpha_2}\xrightarrow{\simeq}V_{\alpha_2},\rho_1^{\beta_2}:V_\beta/W^{\beta_2}\xrightarrow{\simeq}V_{\beta_1},\rho_2^{\beta_2}:W^{\beta_2}\xrightarrow{\simeq}V_{\beta_2}$.

\begin{theorem}\label{main}
For any $A\in \cD^{b,ss}_{\rG_{V_\alpha},m}(\rE_{V_\alpha}), B\in \cD^{b,ss}_{\rG_{V_\beta},m}(\rE_{V_\beta})$, we have
\begin{align*}
&\Res^{\gamma}_{\alpha',\beta'}\Ind^{\gamma}_{\alpha,\beta}(A\boxtimes B)\simeq \\&\bigoplus_{\lambda=(\alpha_1,\alpha_2,\beta_1,\beta_2)\in \cN}\!\!\!\!\!\!\!(\Ind^{\alpha'}_{\alpha_1,\beta_1}\!\boxtimes\! \Ind^{\beta'}_{\alpha_2,\beta_2})(\tau_\lambda)_!((\Res^{\alpha}_{\alpha_1,\alpha_2}A)\!\boxtimes\! (\Res^{\beta}_{\beta_1,\beta_2}B))[-(\alpha_2,\beta_1)](-\frac{(\alpha_2,\beta_1)}{2}).
\end{align*}
\end{theorem}

We only need to show the theorem when $A.B$ are simple perverse sheaves. In this case, we assume that $A$ and $B$ are pure of weight $\omega$ and $\omega'$ respectively.

\subsection{The left hand side}\

We draw the following diagram containing all data we will use, and we explain them later.
\begin{diagram}[midshaft,size=2em]
\rE_{V_\alpha}\times \rE_{V_\beta} &\lTo^{{p_1}^{\gamma}_{\alpha,\beta}} &E'^{\gamma}_{\alpha,\beta} &\rTo^{{p_2}^{\gamma}_{\alpha,\beta}} &E''^{\gamma}_{\alpha,\beta} &\hEq &E''^{\gamma}_{\alpha,\beta} &\rTo^{{p_3}^{\gamma}_{\alpha,\beta}} &\rE_{V_\gamma}\\
 & &\uInto^{\tilde{\iota}_\lambda} &\square&\uInto^{\tilde{\iota}''_\lambda}& &\uInto^{\tilde{\iota}} &\square &\uInto_{\iota^{\gamma}_{\alpha',\beta'}}\\
 & &\cQ_\lambda&\rTo^{\tilde{p_2}_\lambda}&\tilde{F}_\lambda&\rInto^{\tilde{\iota}'_\lambda} &\tilde{F} &\rTo^{\tilde{p_3}} &F^{\gamma}_{\alpha',\beta'}\\
 & &&&\dTo^{f_\lambda} & & & &\dTo_{\kappa^{\gamma}_{\alpha',\beta'}}\\
 & &&&{E''^{\alpha'}_{\alpha_1,\beta_1}\times E''^{\beta'}_{\alpha_2,\beta_2}} & &\rTo^{{p_3}_\lambda}  &&\rE_{V_{\alpha'}}\times \rE_{V_{\beta'}}
\end{diagram}

By definition, the left hand side of the formula is equal to 
$$(\kappa^{\gamma}_{\alpha',\beta'})_!(\iota^{\gamma}_{\alpha',\beta'})^*({p_3}^{\gamma}_{\alpha,\beta})_!({p_2}^{\gamma}_{\alpha,\beta})_\flat({p_1}^{\gamma}_{\alpha,\beta})^*(A\boxtimes B)[M](\frac{M}{2}),$$
where $M=\sum_{h\in H}\alpha_{h'}\beta_{h''}+\sum_{i\in I}\alpha_i\beta_i-\langle\alpha',\beta'\rangle$.

Let $$\tilde{F}=E''^\gamma_{\alpha,\beta}\times_{\rE_{V_\gamma}}F^\gamma_{\alpha',\beta'}$$ 
be the fibre product of ${p_3}^\gamma_{\alpha,\beta}:E''^\gamma_{\alpha,\beta}\rightarrow \rE_{V_\gamma}$ and $\iota^\gamma_{\alpha',\beta'}:F^\gamma_{\alpha',\beta'}\rightarrow  \rE_{V_\gamma}$, that is, there is a Cartesian diagram
\begin{diagram}[midshaft,size=2em]
E''^{\gamma}_{\alpha,\beta} &\rTo^{{p_3}^{\gamma}_{\alpha,\beta}} &\rE_{V_\gamma}\\
\uInto^{\tilde{\iota}} &\Box &\uInto_{\iota^{\gamma}_{\alpha',\beta'}}\\
\tilde{F} &\rTo^{\tilde{p_3}} &F^{\gamma}_{\alpha',\beta'}.
\end{diagram}
By base change, we have $(\iota^{\gamma}_{\alpha',\beta'})^*({p_3}^{\gamma}_{\alpha,\beta})_!\simeq (\tilde{p_3})_!\tilde{\iota}^*$, and so the left hand side of the formula is isomorphic to
$$(\kappa^{\gamma}_{\alpha',\beta'})_!(\tilde{p_3})_!(\tilde{\iota})^*({p_2}^{\gamma}_{\alpha,\beta})_\flat({p_1}^{\gamma}_{\alpha,\beta})^*(A\boxtimes B)[M](\frac{M}{2}).$$

Note that $\tilde{F}$ consists of pairs $(x,W)$, where $x\in \rE_{V_\gamma}$ and $W\subset V_\gamma$ is a $I$-graded subspace of dimension vector $\beta$ such that $W,W^{\beta'}$ are $x$-stable.

For each $\lambda=(\alpha_1,\alpha_2,\beta_1,\beta_2)\in \cN$, we define a locally closed subvariety $\tilde{F}_\lambda\subset \tilde{F}$ consisting of pairs $(x,W)\in \tilde{F}$ such that the dimension vector of $W\cap W^{\beta'}$ is $\beta_2$. Then we define a morphism
\begin{align*}
f_\lambda:&\tilde{F}_\lambda\rightarrow E''^{\alpha'}_{\alpha_1,\beta_1}\times E''^{\beta'}_{\alpha_2,\beta_2}\\
&(x,W)\mapsto ((\rho_{1*}^{\beta'}\overline{x}^{W^{\beta'}},\rho_1^{\beta'}(W/W\cap W^{\beta'})),(\rho_{2*}^{\beta'}x|_{W^{\beta'}},\rho_2^{\beta'}(W\cap W^{\beta'}))),
\end{align*}
where we regard $W/W\cap W^{\beta'}\simeq W+W^{\beta'}/W^{\beta'}$ as a $I$-graded subspace of $V_{\gamma}/W^{\beta'}$ and apply $\rho_1^{\beta'}: V_{\gamma}/W^{\beta'}\rightarrow V_{\alpha'}$ to obtain a $I$-graded subspace $\rho_1^{\beta'}(W/W\cap W^{\beta'})\subset V_{\alpha'}$ of dimension vector $\beta_1$ which is $\rho_{1*}^{\beta'}\overline{x}^{W^{\beta'}}$-stable.

\begin{lemma}\label{vector bundle}
The morphism $f_\lambda$ is a locally trivial vector bundle of rank 
$$L_\lambda=\sum_{h\in H}(\alpha_{1h'}\alpha_{2h''}+\alpha_{1h'}\beta_{2h''}+\beta_{1h'}\beta_{2h''})+\sum_{i\in I}\alpha_{2i}\beta_{1i}$$
such that the following diagram is commutative
\begin{diagram}[midshaft,size=2em]
\tilde{F}_\lambda &\rInto^{\tilde{\iota}'_\lambda} &\tilde{F}\\
\dTo^{f_\lambda} & &\dTo_{\kappa^{\gamma}_{\alpha',\beta'}\tilde{p_3}}\\
E''^{\alpha'}_{\alpha_1,\beta_1}\times E''^{\beta'}_{\alpha_2,\beta_2} &\rTo^{{p_3}_\lambda} &\rE_{V_{\alpha'}}\times \rE_{V_{\beta'}},
\end{diagram}
where ${p_3}_\lambda={p_3}^{\alpha'}_{\alpha_1,\beta_1}\times {p_3}^{\beta'}_{\alpha_2,\beta_2}$ and $\tilde{\iota}'_\lambda$ is the inclusion.
\end{lemma}
\begin{proof}
By definition, it is clear that the diagram commutes.

We fix a $I$-graded direct sum decomposition $V_\gamma=\tilde{W}\oplus W^{\beta'}$ and fix a $I$-graded linear isomorphism $\rho:\tilde{W}\xrightarrow{\simeq} V_\gamma/W^{\beta'}$.

For each fixed $(x',W')\in E''^{\alpha'}_{\alpha_1,\beta_1}$ and $(x'',W'')\in E''^{\beta'}_{\alpha_2,\beta_2}$, recall that $x'\in \rE_{V_{\alpha'}},x''\in \rE_{V_{\beta'}}$ and $W'\subset V_{\alpha'}$ is a $x'$-stable $I$-graded subspace of dimension vector $\beta_1$, $W''\subset V_{\beta'}$ is a $x''$-stable $I$-graded subspace of dimension vector $\beta_2$. Let $\Gamma$ be the fibre of $f_\lambda$ at $((x',W'),(x'',W''))$.

Firstly, giving a $I$-graded subspace $W\subset V_{\gamma}$ of dimension vector $\beta$ such that the dimension vector of $W\cap W^{\beta'}$ is $\beta_2$ and such that $\rho_1^{\beta'}(W/W\cap W^{\beta'})=W',\rho_2^{\beta'}(W\cap W^{\beta'})=W''$ is the same as giving a $I$-graded linear map $z:W'\rightarrow V_{\beta'}/W''$. Indeed, the map $z$ corresponds to the space
$$W=\{v'+v''\in\tilde{W}\oplus W^{\beta'}\mid\rho_1^{\beta'}\rho(v')\in W',z\rho_1^{\beta'}\rho(v')=\rho_2^{\beta'}(v'')+W''\}.$$

Secondly, giving $x\in \rE_{V_\gamma}$ such that $W^{\beta'}$ is $x$-stable and $\rho_{1*}^{\beta'}\overline{x}^{W^{\beta'}}=x', \rho_{2*}^{\beta'}x|_{W^{\beta'}}=x''$ is the same as giving $y\in \bigoplus_{h\in H}\Hom_k((V_{\alpha'})_{h'},(V_{\beta'})_{h''})$. Indeed, the element $y$ corresponds to the element
$$\begin{pmatrix}
\rho_1^{\beta'}\rho &\\
 &\rho_2^{\beta'}\end{pmatrix}^{-1}\begin{pmatrix}
 x' &0\\
 y &x''\end{pmatrix}\begin{pmatrix}
\rho_1^{\beta'}\rho &\\
 &\rho_2^{\beta'}\end{pmatrix},$$
where we write elements of $\rE_{V_\gamma}$ as block matrices with respect to the decomposition $V_\gamma=\tilde{W}\oplus W^{\beta'}$.

Finally, under above correspondences, the condition $W$ is $x$-stable is the same as the condition
$$(z_{h''}x'_h-(\overline{x''}^{W''})_hz_{h'}-\pi_{h''} y_h)(W'_{h'})=0$$
for any $h\in H$, where $\pi_{h''}:(V_{\beta'})_{h''}\rightarrow (V_{\beta'}/W'')_{h''}$ is the quotient map.

Hence $\Gamma$ can be identified with the linear space $\Gamma'$ consisting of pairs $(z,y)$, where $y\in \bigoplus_{h\in H}\Hom_k((V_{\alpha'})_{h'},(V_{\beta'})_{h''})$ and $z:W'\rightarrow V_{\beta'}/W''$ is a $I$-graded linear map satisfying above condition. Let $\Gamma'_1$ be a subspace of $\bigoplus_{h\in H}\Hom_k((V_{\alpha'})_{h'},(V_{\beta'})_{h''})$ consisting of $y$ satisfying $y_h(W'_{h'})\subset W''_{h''}$ for any $h\in H$, and let $\Gamma'_2$ be the linear space consisting of all $I$-graded linear maps $z:W'\rightarrow V_{\beta'}/W''$. Then there is a short exact sequence
$$
0\rightarrow \Gamma'_1\xrightarrow{s} \Gamma\xrightarrow{t} \Gamma'_2\rightarrow 0,
$$
where $s(y)=(0,y)$ and $t(z,y)=z$. Therefore 
\begin{align*}
&\dim \Gamma=\dim \Gamma'=\dim \Gamma'_1+\dim \Gamma'_2\\
=&\sum_{h\in H}\dim W'_{h'}\dim W''_{h''}+\dim (V_{\alpha'}/W')_{h'}\dim (V_{\beta'})_{h''}\\
&+\sum_{i\in I}\dim W'_i\dim (V_{\beta'}/W'')_i\\
=&\sum_{h\in H}(\beta_{1h'}\beta_{2h''}+\alpha_{1h'}\beta'_{h''})+\sum_{i\in I}\alpha_{2i}\beta_{1i}\\
=&\sum_{h\in H}(\alpha_{1h'}\alpha_{2h''}+\alpha_{1h'}\beta_{2h''}+\beta_{1h'}\beta_{2h''})+\sum_{i\in I}\alpha_{2i}\beta_{1i}=L_\lambda,
\end{align*}
as desired.
\end{proof}

All these locally closed subvarieties $\tilde{F}_\lambda$ form a partition 
$$\tilde{F}=\bigsqcup_{\lambda\in \cN}\tilde{F}_\lambda.$$

For each $n\in \mathbb{Z}_{\geqslant 0}$, let $\tilde{F}_n$ be the disjoint union of all $\tilde{F}_\lambda$ which have dimension $n$, then these $\tilde{F}_\lambda$ are both open and closed in $\tilde{F}_n$. For $n\in \mathbb{Z}_{<0}$, let $\tilde{F}_n=\varnothing$ for convenience. 

Let $\cN_n=\{\lambda\in\cN\mid \tilde{F}_\lambda\subset \tilde{F}_n\}$ and $E''_n$ be the disjoint union of $E''^{\alpha'}_{\alpha_1,\beta_1}\times E''^{\beta'}_{\alpha_2,\beta_2}$ with $\lambda=(\alpha_1,\alpha_2,\beta_1,\beta_2)\in \cN_n$. Then there is a commutative diagram
\begin{diagram}[midshaft,size=2em]
\tilde{F}_n &\rInto^{\iota_n} &\tilde{F}\\
\dTo^{f_n} & &\dTo_{\kappa^{\gamma}_{\alpha',\beta'}\tilde{p_3}}\\
E''_n&\rTo^{{p_3}_n} &\rE_{V_{\alpha'}}\times \rE_{V_{\beta'}},
\end{diagram}
where $\iota_n,f_n$ and ${p_3}_n$ are assembled by $\tilde{\iota}'_\lambda,f_\lambda$ and ${p_3}_\lambda$ with $\lambda=(\alpha_1,\alpha_2,\beta_1,\beta_2)\in \cN_n$ respectively.

\begin{lemma}\label{n-case}
With the same notation as above, we have  
\begin{align*}
&(\kappa^{\gamma}_{\alpha',\beta'})_!(\tilde{p_3})_!(\iota_n)_!(\iota_n)^*(\tilde{\iota})^*({p_2}^{\gamma}_{\alpha,\beta})_\flat({p_1}^{\gamma}_{\alpha,\beta})^*(A\boxtimes B)\\
\simeq &\bigoplus_{\lambda=(\alpha_1,\alpha_2,\beta_1,\beta_2)\in \cN_n}({p_3}_\lambda)_!(f_\lambda)_!(\tilde{\iota}''_\lambda)^*({p_2}^{\gamma}_{\alpha,\beta})_\flat({p_1}^{\gamma}_{\alpha,\beta})^*(A\boxtimes B),
\end{align*}
where $\tilde{\iota}''_\lambda=\tilde{\iota}\tilde{\iota}'_\lambda:\tilde{F}_\lambda\hookrightarrow E''^{\gamma}_{\alpha,\beta}$ is the inclusion.
\end{lemma}
\begin{proof}
We denote by $C=(\tilde{\iota})^*({p_2}^\gamma_{\alpha,\beta})_\flat({p_1}^\gamma_{\alpha,\beta})^*(A\boxtimes B), \theta_n=\kappa^\gamma_{\alpha',\beta'}\tilde{p_3}\iota_n={p_3}_nf_n:\tilde{F}_n\rightarrow \rE_{V_{\alpha'}}\times \rE_{V_{\beta'}}$ and $\theta_\lambda=\kappa^\gamma_{\alpha',\beta'}\tilde{p_3}\tilde{\iota}'_\lambda={p_3}_\lambda f_\lambda:\tilde{F}_\lambda\rightarrow \rE_{V_{\alpha'}}\times \rE_{V_{\beta'}}$ for any $\lambda\in \cN_n$. Note that all $\tilde{F}_\lambda$ are both open and closed in $\tilde{F}_n$ for $\lambda\in\cN_n$, by \uppercase\expandafter{\romannumeral7}.Theorem 1.4 and \uppercase\expandafter{\romannumeral3}.Mayer-Vietoris sequence 1.8 in \cite{Iversen-1986}, we have
\begin{align*}
(\theta_n)_!(\iota_n)^*(C)\simeq \bigoplus_{\lambda\in \cN_n}(\theta_\lambda)_!(\tilde{\iota}'_\lambda)^*(C),
\end{align*}
as desired.
\end{proof}

Recall that the group $Q^{\beta'}$ acts on $\bigsqcup_{\lambda\in \cN}E''^{\alpha'}_{\alpha_1,\beta_1}\times E''^{\beta'}_{\alpha_2,\beta_2}$ through the quotient $Q^{\beta'}/U^{\beta'}\simeq \rG_{V_{\alpha'}}\times \rG_{V_{\beta'}}$, where $U^{\beta'}$ is the unipotent radical which acts trivially on $\bigsqcup_{\lambda\in \cN}E''^{\alpha'}_{\alpha_1,\beta_1}\times E''^{\beta'}_{\alpha_2,\beta_2}$, by Proposition \ref{sub-quotient-equivariant}, there is an equivalence
$$\cD^b_{Q^{\beta'},m}(\bigsqcup_{\lambda\in \cN}E''^{\alpha'}_{\alpha_1,\beta_1}\times E''^{\beta'}_{\alpha_2,\beta_2})\xrightarrow{\simeq} \cD^b_{\rG_{V_{\alpha'}}\times \rG_{V_{\beta'}},m}(\bigsqcup_{\lambda\in \cN}E''^{\alpha'}_{\alpha_1,\beta_1}\times E''^{\beta'}_{\alpha_2,\beta_2}).$$

\begin{lemma}\label{hyperbolic-2}
The functor 
\begin{align*}
&\cD^b_{\rG_{V_\gamma},m}(E''^\gamma_{\alpha,\beta})\xrightarrow{\textrm{forget}}\cD^b_{Q^{\beta'},m}(E''^\gamma_{\alpha,\beta})\xrightarrow{(\tilde{\iota})^*}\cD^b_{Q^{\beta'},m}(\tilde{F})\\&\xrightarrow{f_!}\cD^b_{Q^{\beta'},m}(\bigsqcup_{\lambda\in \cN}E''^{\alpha'}_{\alpha_1,\beta_1}\times E''^{\beta'}_{\alpha_2,\beta_2})\xrightarrow{\simeq} \cD^b_{\rG_{V_{\alpha'}}\times \rG_{V_{\beta'}},m}(\bigsqcup_{\lambda\in \cN}E''^{\alpha'}_{\alpha_1,\beta_1}\times E''^{\beta'}_{\alpha_2,\beta_2}),
\end{align*}
sends pure complexes of weight $\eta$ to pure complexes of weight $\eta$, where the morphism $f:\tilde{F}\rightarrow \bigsqcup_{\lambda\in \cN}E''^{\alpha'}_{\alpha_1,\beta_1}\times E''^{\beta'}_{\alpha_2,\beta_2}$ is assembled by $f_\lambda$ with $\lambda\in \cN$.
\end{lemma}
\begin{proof}
We claim that the functor 
$$\cD^b_{Q^{\beta'},m}(E''^\gamma_{\alpha,\beta})\xrightarrow{(\tilde{\iota})^*}\cD^b_{m}(\tilde{F})\xrightarrow{f_!}\cD^b_{m}(\bigsqcup_{\lambda\in \cN}E''^{\alpha'}_{\alpha_1,\beta_1}\times E''^{\beta'}_{\alpha_2,\beta_2})$$
is a hyperbolic localization functor. For any $\lambda=(\alpha_1,\alpha_2,\beta_1,\beta_2)\in \cN$, let $E''_\lambda$ be the locally closed subvariety of $E''^\gamma_{\alpha,\beta}$ consisting of $(x,W)$ such that the dimension vector of $W\cap W^{\beta'}$ is $\beta_2$, then all these closed subvarieties form a partition 
$$E''^\gamma_{\alpha,\beta}=\bigsqcup_{\lambda\in \cN}E''_\lambda.$$
We fix a $I$-graded direct sum decomposition $V_\gamma=\tilde{W}\oplus W^{\beta'}$ and fix a $I$-graded linear isomorphism $\rho:\tilde{W}\xrightarrow{\simeq}V_\gamma/W^{\beta'}$. By the proof of Proposition \ref{res-hyperbolic}, there is a bijection
\begin{align*}
\rE_{V_{\nu'}}\!\times\! \rE_{V_{\nu''}}&\!\times\! \bigoplus_{h\in H}\Hom_k((V_{\nu'})_{h'},(V_{\nu''})_{h''})\!\times\! \bigoplus_{h\in H}\Hom_k((V_{\nu''})_{h'},(V_{\nu'})_{h''})\simeq \rE_{V_\nu}\\
&(x',x'',y,w)\mapsto\begin{pmatrix}
\rho_1^{\beta'}\rho &\\
 &\rho_2^{\beta'}\end{pmatrix}^{-1}\begin{pmatrix}
 x' &w\\
 y &x''\end{pmatrix}\begin{pmatrix}
\rho_1^{\beta'}\rho &\\
 &\rho_2^{\beta'}\end{pmatrix}.
\end{align*}
For any $\lambda=(\alpha_1,\alpha_2,\beta_1,\beta_2)\in \cN$, let $Gr_\lambda$ be the subvariety of the $I$-graded Grassmannian $Gr_{\beta}(V_\gamma)$ consisting of $I$-graded subspaces $W\subset V_\gamma$ of dimension vector $\beta$ such that the dimension vector of $W\cap W^{\beta'}$ is $\beta_2$, then there is a partition 
$$Gr_{\beta}(V_\gamma)=\bigsqcup_{\lambda\in \cN}Gr_\lambda$$
and a vector bundle
\begin{align*}
Gr_\lambda&\rightarrow Gr_{\beta_1}(V_{\alpha'})\times Gr_{\beta_2}(V_{\beta'})\\
W&\mapsto (\rho_1^{\beta'}(W/W\cap W^{\beta'}),\rho_2^{\beta'}(W\cap W^{\beta'}))
\end{align*}
whose fibre at $(W',W'')\in Gr_{\beta_1}(V_{\alpha'})\times Gr_{\beta_2}(V_{\beta'})$ can be identified with the linear space of $I$-graded linear maps $z:W'\rightarrow V_{\beta'}/W''$. Indeed, the map $z$ corresponds to the space
$$W_{W',W'',z}=\{v'+v''\in \tilde{W}\oplus W^{\beta'}\mid \rho_1^{\beta'}\rho(v')\in W',z\rho_1^{\beta'}\rho(v')=\rho_2^{\beta'}(v'')+W''\}.$$
Under above correspondences, an element $(x,W)\in E''_\lambda\subset\rE_{V_{\gamma}}\times Gr_\lambda$ is identified with $(x',x'',y,w,W',W'',z)$ satisfying
\begin{align*}
&x'\rho_1^{\beta'}\rho(v')+w\rho_2^{\beta'}(v'')\in W',\\
&z(x'\rho_1^{\beta'}\rho(v')+w\rho_2^{\beta'}(v''))=y\rho_1^{\beta'}\rho(v')+x''\rho_2^{\beta'}(v'')+W''
\end{align*}
for any $v'+v''\in W_{W',W'',z}$. There is a one-parameter subgroup embedding $\zeta:k^*\hookrightarrow Q^{\beta'}$ given by $t\mapsto \textrm{Id}_{V_{\alpha'}}\bigoplus t\textrm{Id}_{V_{\beta'}}$ such that $k^*$ acts on $E''_\lambda$ via $\zeta$. More precisely, 
the $k^*$-action on $E''_\lambda$ is given by
$$t.(x',x'',y,w,W',W'',z)=(x',x'',ty,t^{-1}w,W',W'',tz).$$
Then there is a commutative diagram
\begin{diagram}[midshaft,size=2em]
(E''_\lambda)^{k^*} &\lTo^{\pi_\lambda^+} &(E''_\lambda)^+ &\rInto^{g_\lambda^+} &E''_\lambda\\
\dTo^{\simeq} & &\dTo^{\simeq} & &\vEq\\
E''^{\alpha'}_{\alpha_1,\beta_1}\times E''^{\beta'}_{\alpha_2,\beta_2} &\lTo^{f_\lambda} &\tilde{F}_\lambda &\rInto^{i_\lambda} &E''_\lambda,
\end{diagram}
where ${i_\lambda}:\tilde{F}_\lambda\hookrightarrow E''_\lambda$ is the inclusion. Assembling them together, the following diagram commutes
\begin{diagram}[midshaft,size=2em]
(E''^\gamma_{\alpha,\beta})^{k^*} &\lTo^{\pi^+} &(E''^\gamma_{\alpha,\beta})^+ &\rInto^{g^+} &E''^\gamma_{\alpha,\beta}\\
\dTo^{\simeq} & &\dTo^{\simeq} & &\vEq\\
\bigsqcup_{\lambda\in \cN}E''^{\alpha'}_{\alpha_1,\beta_1}\times E''^{\beta'}_{\alpha_2,\beta_2} &\lTo^{f} &\tilde{F} &\rInto^{\tilde{\iota}} &E''^\gamma_{\alpha,\beta},
\end{diagram}
where $\pi^+,g^+$ are assembled by $\pi_\lambda^+,g_\lambda^+$ with $\lambda\in \cN$ respectively.
Hence the functor $f_!(\tilde{\iota})^*:\cD^b_{Q^{\beta'},m}(E''^\gamma_{\alpha,\beta})\rightarrow \cD^b_{m}(\bigsqcup_{\lambda\in \cN}E''^{\alpha'}_{\alpha_1,\beta_1}\times E''^{\beta'}_{\alpha_2,\beta_2})$ is a hyperbolic localization functor, and so it sends pure complexes of weight $\eta$ to pure complexes of weight $\eta$ by Theorem \ref{hyperbolic}.
\end{proof}

\begin{corollary}\label{lambda-pure}
For any $\lambda=(\alpha_1,\alpha_2,\beta_1,\beta_2)\in \cN$, the complex 
$$({p_3}_\lambda)_!(f_\lambda)_!(\tilde{\iota}''_\lambda)^*({p_2}^{\gamma}_{\alpha,\beta})_\flat({p_1}^{\gamma}_{\alpha,\beta})^*(A\boxtimes B)$$
is pure of weight $\omega+\omega'$, and so it is semisimple in $\cD^b_{\rG_{V_{\alpha'}}\times \rG_{V_{\beta'}},m}(\rE_{V_{\alpha'}}\times \rE_{V_{\beta'}})$, where $A,B$ are simple perverse sheaves which are pure of weight $\omega,\omega'$ respectively in the beginning.
\end{corollary}
\begin{proof}
Since $A,B$ are pure of weight $\omega,\omega'$ respectively, we know that $A\boxtimes B$ is pure of weight $\omega+\omega'$ by Section 5.1.14 in \cite{Beilinson-Bernstein-Deligne-1982}. Then by Proposition \ref{pure-*!} and Corollary \ref{principal}, the complex 
$({p_2}^{\gamma}_{\alpha,\beta})_\flat({p_1}^{\gamma}_{\alpha,\beta})^*(A\boxtimes B)$
is pure of weight $\omega+\omega'$, and so the complex
$$f_!(\tilde{\iota}')^*({p_2}^{\gamma}_{\alpha,\beta})_\flat({p_1}^{\gamma}_{\alpha,\beta})^*(A\boxtimes B)$$
is pure of weight $\omega+\omega'$ by Lemma \ref{hyperbolic-2} which is semisimple complexs on 
$$\bigsqcup_{\lambda\in \cN}E''^{\alpha'}_{\alpha_1,\beta_1}\times E''^{\beta'}_{\alpha_2,\beta_2}.$$
Note that for any $\lambda=(\alpha_1,\alpha_2,\beta_1,\beta_2)\in \cN$, the subvariety $E''^{\alpha'}_{\alpha_1,\beta_1}\times E''^{\beta'}_{\alpha_2,\beta_2}$ is closed in $\tilde{F}_\lambda$, since it is zero section of the vector bundle $f_\lambda:\tilde{F}_\lambda\rightarrow E''^{\alpha'}_{\alpha_1,\beta_1}\times E''^{\beta'}_{\alpha_2,\beta_2}$. Thus each component is closed in $\bigsqcup_{\lambda\in \cN}E''^{\alpha'}_{\alpha_1,\beta_1}\times E''^{\beta'}_{\alpha_2,\beta_2}$, moreover, it is both open and closed, since there are finitely many components. Hence the restriction of $f_!(\tilde{\iota}')^*({p_2}^{\gamma}_{\alpha,\beta})_\flat({p_1}^{\gamma}_{\alpha,\beta})^*(A\boxtimes B)$ on each component $E''^{\alpha'}_{\alpha_1,\beta_1}\times E''^{\beta'}_{\alpha_2,\beta_2}$ is pure of weight $\omega+\omega'$ which is equal to 
$$(f_\lambda)_!(\tilde{\iota}''_\lambda)^*({p_2}^{\gamma}_{\alpha,\beta})_\flat({p_1}^{\gamma}_{\alpha,\beta})^*(A\boxtimes B).$$
Finally, by Proposition \ref{pure-*!} and Corollary \ref{semisimple-*!}, 
$$({p_3}_\lambda)_!(f_\lambda)_!(\tilde{\iota}''_\lambda)^*({p_2}^{\gamma}_{\alpha,\beta})_\flat({p_1}^{\gamma}_{\alpha,\beta})^*(A\boxtimes B)$$
is pure of weight $\omega+\omega'$, and is a semisimple complex.
\end{proof}

For each $n\in \bbZ_{\geqslant 0}$, let $\tilde{F}_{\leqslant n}$ be the union of $\tilde{F}_{n'}$ with $n'\leqslant n$, then $\tilde{F}_{\leqslant n-1}$ is closed in $\tilde{F}_{\leqslant n}$ and $\tilde{F}_{\leqslant n}\setminus \tilde{F}_{\leqslant n-1}=\tilde{F}_n$. Let $\iota_{\leqslant n}:\tilde{F}_{\leqslant n}\rightarrow \tilde{F}$ be the inclusion, by Section 1.10 in \cite{Lusztig-1985}, there is a distinguished triangle
$$(\iota_n)_!(\iota_n)^*(C)\rightarrow (\iota_{\leqslant n})_!(\iota_{\leqslant n})^*(C)\rightarrow(\iota_{\leqslant n-1})_!(\iota_{\leqslant n-1})^*(C)\rightarrow (\iota_n)_!(\iota_n)^*(C)[1]$$
in $\cD^b_{Q^{\beta'},m}(\tilde{F})$, where $C=(\tilde{\iota})^*({p_2}^{\gamma}_{\alpha,\beta})_\flat({p_1}^{\gamma}_{\alpha,\beta})^*(A\boxtimes B)$. Applying triangulated functor $(\kappa^{\gamma}_{\alpha',\beta'})_!(\tilde{p_3})_!$, we obtain a distinguished triangle  
\begin{align*}
&(\kappa^{\gamma}_{\alpha',\beta'})_!(\tilde{p_3})_!(\iota_n)_!(\iota_n)^*(C)\rightarrow (\kappa^{\gamma}_{\alpha',\beta'})_!(\tilde{p_3})_!(\iota_{\leqslant n})_!(\iota_{\leqslant n})^*(C)\\
&\rightarrow(\kappa^{\gamma}_{\alpha',\beta'})_!(\tilde{p_3})_!(\iota_{\leqslant n-1})_!(\iota_{\leqslant n-1})^*(C)\rightarrow (\kappa^{\gamma}_{\alpha',\beta'})_!(\iota_n)_!(\iota_n)^*(C)[1]
\end{align*}
in $\cD^b_{\rG_{V_{\alpha'}}\times\rG_{V_{\beta'}},m}(\rE_{V_{\alpha'}}\times \rE_{V_{\beta'}})$. Applying perverse cohomology functors, we obtain a long exact sequence of perverse sheaves
\begin{align*}
&...\rightarrow\lp H^{s-1}((\kappa^{\gamma}_{\alpha',\beta'})_!(\tilde{p_3})_!(\iota_{\leqslant n-1})_!(\iota_{\leqslant n-1})^*(C))\xrightarrow{\delta_s}\lp H^s((\kappa^{\gamma}_{\alpha',\beta'})_!(\tilde{p_3})_!(\iota_n)_!(\iota_n)^*(C))\\
&\rightarrow \lp H^s((\kappa^{\gamma}_{\alpha',\beta'})_!(\tilde{p_3})_!(\iota_{\leqslant n})_!(\iota_{\leqslant n})^*(C))\rightarrow\lp H^s((\kappa^{\gamma}_{\alpha',\beta'})_!(\tilde{p_3})_!(\iota_{\leqslant n-1})_!(\iota_{\leqslant n-1})^*(C))\rightarrow...
\end{align*}

\begin{lemma}\label{inductive}With the same notations as above, we have\\
(a) for any for $n\in \bbZ_{\geqslant 0}$ and $s\in \bbZ$, the connecting morphism $\delta_s=0$;\\
(b) for any $n\in \bbZ_{\geqslant 0}$, the complex $(\kappa^{\gamma}_{\alpha',\beta'})_!(\tilde{p_3})_!(\iota_{\leqslant n})_!(\iota_{\leqslant n})^*(C)$ is pure of weight $\omega+\omega'$ and 
\begin{align*}
&(\kappa^{\gamma}_{\alpha',\beta'})_!(\tilde{p_3})_!(\iota_{\leqslant n})_!(\iota_{\leqslant n})^*(C)\\
\simeq &(\kappa^{\gamma}_{\alpha',\beta'})_!(\tilde{p_3})_!(\iota_{\leqslant n-1})_!(\iota_{\leqslant n-1})^*(C)\oplus (\kappa^{\gamma}_{\alpha',\beta'})_!(\tilde{p_3})_!(\iota_n)_!(\iota_n)^*(C).
\end{align*}
\end{lemma}
\begin{proof}
We make an induction on $n$ to show (a) and (b). Note that if $n< 0$, we have $\tilde{F}_{\leqslant n}=\varnothing$ and the complex $(\kappa^{\gamma}_{\alpha',\beta'})_!(\tilde{p_3})_!(\iota_{\leqslant n})_!(\iota_{\leqslant n})^*(C)=0$, and so the statements are trivial. 

Assume that statements hold for $n-1$, then for any $s\in \bbZ$, by Corollary 5.4.4 in \cite{Beilinson-Bernstein-Deligne-1982}, the perverse sheaf $\lp H^{s-1}((\kappa^{\gamma}_{\alpha',\beta'})_!(\tilde{p_3})_!(\iota_{\leqslant n-1})_!(\iota_{\leqslant n-1})^*(C))$ is pure of weight $\omega+\omega'+s-1$. By Lemma \ref{n-case} and Corollary \ref{lambda-pure}, we know that the complexes
\begin{align*}
&(\kappa^{\gamma}_{\alpha',\beta'})_!(\tilde{p_3})_!(\iota_n)_!(\iota_n)^*(C)\\
\simeq &\bigoplus_{\lambda=(\alpha_1,\alpha_2,\beta_1,\beta_2)\in \cN_n}({p_3}_\lambda)_!(f_\lambda)_!(\tilde{\iota}''_\lambda)^*({p_2}^{\gamma}_{\alpha,\beta})_\flat({p_1}^{\gamma}_{\alpha,\beta})^*(A\boxtimes B)
\end{align*}
is pure of weight $\omega+\omega'$, so the perverse sheaf $\lp H^s((\kappa^{\gamma}_{\alpha',\beta'})_!(\tilde{p_3})_!(\iota_n)_!(\iota_n)^*(C))$ is pure of weight $\omega+\omega'+s$ by Corollary 5.4.4 in \cite{Beilinson-Bernstein-Deligne-1982}. Thus $\delta_s$ is a morphism between two perverse sheaves of different weights which must be zero, as desired. 

Then the long exact sequence decomposes into many short exact sequence 
\begin{align*}
0\rightarrow&\lp H^s((\kappa^{\gamma}_{\alpha',\beta'})_!(\tilde{p_3})_!(\iota_n)_!(\iota_n)^*(C))\rightarrow \lp H^s((\kappa^{\gamma}_{\alpha',\beta'})_!(\tilde{p_3})_!(\iota_{\leqslant n})_!(\iota_{\leqslant n})^*(C))\\
\rightarrow &\lp H^s((\kappa^{\gamma}_{\alpha',\beta'})_!(\tilde{p_3})_!(\iota_{\leqslant n-1})_!(\iota_{\leqslant n-1})^*(C))\rightarrow 0,
\end{align*}
where $\lp H^s((\kappa^{\gamma}_{\alpha',\beta'})_!(\tilde{p_3})_!(\iota_n)_!(\iota_n)^*(C))$ and $\lp H^s((\kappa^{\gamma}_{\alpha',\beta'})_!(\tilde{p_3})_!(\iota_{\leqslant n-1})_!(\iota_{\leqslant n-1})^*(C))$ are pure of weight of $\omega+\omega'+s$, thus so is $\lp H^s((\kappa^{\gamma}_{\alpha',\beta'})_!(\tilde{p_3})_!(\iota_{\leqslant n})_!(\iota_{\leqslant n})^*(C))$. Hence the complex $(\kappa^{\gamma}_{\alpha',\beta'})_!(\tilde{p_3})_!(\iota_{\leqslant n})_!(\iota_{\leqslant n})^*(C)$ is pure of weight $\omega+\omega'$ by Corollary 5.4.4 in \cite{Beilinson-Bernstein-Deligne-1982}. 

By Theorem \ref{BBD}, above exact sequences are short exact sequences of semisimple perverse sheaves which must split, and so 
\begin{align*}
&\lp H^s((\kappa^{\gamma}_{\alpha',\beta'})_!(\tilde{p_3})_!(\iota_{\leqslant n})_!(\iota_{\leqslant n})^*(C))\\
\simeq &\lp H^s((\kappa^{\gamma}_{\alpha',\beta'})_!(\tilde{p_3})_!(\iota_{\leqslant n-1})_!(\iota_{\leqslant n-1})^*(C))\oplus \lp H^s((\kappa^{\gamma}_{\alpha',\beta'})_!(\tilde{p_3})_!(\iota_n)_!(\iota_n)^*(C))
\end{align*}
for any $s\in \bbZ$. By Theorem 5.4.5 in \cite{Beilinson-Bernstein-Deligne-1982}, we have 
\begin{align*}
&(\kappa^{\gamma}_{\alpha',\beta'})_!(\tilde{p_3})_!(\iota_{\leqslant n})_!(\iota_{\leqslant n})^*(C)\\
\simeq &(\kappa^{\gamma}_{\alpha',\beta'})_!(\tilde{p_3})_!(\iota_{\leqslant n-1})_!(\iota_{\leqslant n-1})^*(C)\oplus (\kappa^{\gamma}_{\alpha',\beta'})_!(\tilde{p_3})_!(\iota_n)_!(\iota_n)^*(C),
\end{align*}
as desired.
\end{proof}

\begin{corollary}\label{Lusztig-inductive}
The left hand side of the formula in Theorem \ref{main} is isomorphic to 
\begin{align*}
&(\kappa^{\gamma}_{\alpha',\beta'})_!(\tilde{p_3})_!(\tilde{\iota})^*({p_2}^{\gamma}_{\alpha,\beta})_\flat({p_1}^{\gamma}_{\alpha,\beta})^*(A\boxtimes B)[M](\frac{M}{2})\\\simeq&\bigoplus_{\lambda=(\alpha_1,\alpha_2,\beta_1,\beta_2)\in \cN}({p_3}_\lambda)_!(f_\lambda)_!(\tilde{\iota}''_\lambda)^*({p_2}^{\gamma}_{\alpha,\beta})_\flat({p_1}^{\gamma}_{\alpha,\beta})^*(A\boxtimes B)[M](\frac{M}{2}).
\end{align*}
\end{corollary}
\begin{proof}
Note that if $n$ is large enough, $\tilde{F}_{\leqslant n}=\tilde{F}$ and $\iota_{\leqslant n}$ is the identity morphism; if $n<0$, $\tilde{F}_{\leqslant n}=\varnothing$. Take $n$ large enough, then $\cN$ is the union of $\cN_{n'}$ with $0\leqslant n'\leqslant n$. By Lemma \ref{n-case} and Lemma \ref{inductive}, we have
\begin{align*}
&(\kappa^{\gamma}_{\alpha',\beta'})_!(\tilde{p_3})_!(\tilde{\iota})^*({p_2}^{\gamma}_{\alpha,\beta})_\flat({p_1}^{\gamma}_{\alpha,\beta})^*(A\boxtimes B)[M](\frac{M}{2})\\=&(\kappa^{\gamma}_{\alpha',\beta'})_!(\tilde{p_3})_!(\iota_{\leqslant n})_!(\iota_{\leqslant n})^*(\tilde{\iota})^*({p_2}^{\gamma}_{\alpha,\beta})_\flat({p_1}^{\gamma}_{\alpha,\beta})^*(A\boxtimes B)[M](\frac{M}{2})\\
\simeq&\bigoplus_{0\leqslant n'\leqslant n}(\kappa^{\gamma}_{\alpha',\beta'})_!(\tilde{p_3})_!(\iota_{n'})_!(\iota_{n'})^*(\tilde{\iota})^*({p_2}^{\gamma}_{\alpha,\beta})_\flat({p_1}^{\gamma}_{\alpha,\beta})^*(A\boxtimes B)[M](\frac{M}{2})\\
\simeq&\bigoplus_{0\leqslant n'\leqslant n}\bigoplus_{\lambda=(\alpha_1,\alpha_2,\beta_1,\beta_2)\in \cN_{n'}}({p_3}_\lambda)_!(f_\lambda)_!(\tilde{\iota}''_\lambda)^*({p_2}^{\gamma}_{\alpha,\beta})_\flat({p_1}^{\gamma}_{\alpha,\beta})^*(A\boxtimes B)[M](\frac{M}{2})\\
=&\bigoplus_{\lambda=(\alpha_1,\alpha_2,\beta_1,\beta_2)\in \cN}({p_3}_\lambda)_!(f_\lambda)_!(\tilde{\iota}''_\lambda)^*({p_2}^{\gamma}_{\alpha,\beta})_\flat({p_1}^{\gamma}_{\alpha,\beta})^*(A\boxtimes B)[M](\frac{M}{2}),
\end{align*}
as desired.
\end{proof}

For each $\lambda=(\alpha_1,\alpha_2,\beta_1,\beta_2)\in\cN$, let 
$$\cQ_\lambda=E'^{\gamma}_{\alpha,\beta}\times_{E''^{\gamma}_{\alpha,\beta}}\tilde{F}_\lambda$$
be the fibre product of ${p_2}^{\gamma}_{\alpha,\beta}:E'^{\gamma}_{\alpha,\beta}\rightarrow E''^{\gamma}_{\alpha,\beta}$ and $\tilde{\iota}''_\lambda:\tilde{F}_\lambda\hookrightarrow E''^{\gamma}_{\alpha,\beta}$, that is, there is a Cartesian diagram
\begin{diagram}[midshaft,size=2em]
E'^{\gamma}_{\alpha,\beta} &\rTo^{{p_2}^{\gamma}_{\alpha,\beta}} &E''^{\gamma}_{\alpha,\beta}\\
\uInto^{\tilde{\iota}_\lambda} &\Box &\uInto_{\tilde{\iota}''_\lambda}\\
\cQ_\lambda &\rTo^{\tilde{p_2}} &\tilde{F}_\lambda.
\end{diagram}

\begin{lemma}\label{p2-iota}
We have $(\tilde{\iota}''_\lambda)^*({p_2}^{\gamma}_{\alpha,\beta})_\flat\simeq (\tilde{p_2})_\flat(\tilde{\iota}_\lambda)^*$.
\end{lemma} 
\begin{proof}
Since ${p_2}^{\gamma}_{\alpha,\beta}$ and $\tilde{p_2}$ are principal $\rG_{V_\alpha}\times \rG_{V_\beta}$-bundle, by Proposition \ref{Lusztig-principal}, $({p_2}^{\gamma}_{\alpha,\beta})_\flat, (\tilde{p_2})_\flat$ have quasi-inverses $({p_2}^{\gamma}_{\alpha,\beta})^*, (\tilde{p_2})^*$ respectively. By above commutative diagram, we have $(\tilde{p_2})^*(\tilde{\iota}''_\lambda)^*=(\tilde{\iota}_\lambda)^*({p_2}^{\gamma}_{\alpha,\beta})^*$, and so $(\tilde{\iota}''_\lambda)^*({p_2}^{\gamma}_{\alpha,\beta})_\flat\simeq (\tilde{p_2})_\flat(\tilde{\iota}_\lambda)^*$.
\end{proof}

By Corollary \ref{Lusztig-inductive} and Lemma \ref{p2-iota}, we obtain the following proposition.

\begin{proposition}\label{left}
The left hand side of the formula in Theorem \ref{main} is isomorphic to 
$$\bigoplus_{\lambda=(\alpha_1,\alpha_2,\beta_1,\beta_2)\in \cN}({p_3}_\lambda)_!(f_\lambda)_!(\tilde{p_2})_\flat(\tilde{\iota}_\lambda)^*({p_1}^{\gamma}_{\alpha,\beta})^*(A\boxtimes B)[M](\frac{M}{2}).$$
\end{proposition}

\subsection{The right hand side}\

We draw the following diagram containing all data we will use, and we explain them later.
\begin{diagram}[midshaft,size=2em]
\rE_{V_\alpha}\times \rE_{V_\beta}\\
\uInto^{\iota_\lambda}\\
F^{\alpha}_{\alpha_1,\alpha_2}\times F^{\beta}_{\beta_1,\beta_2} &\lTo^{\tilde{p_1}_\lambda} &\cO_\lambda\\
\dTo^{\kappa_\lambda}& &\\
\rE_{V_{\alpha_1}}\times \rE_{V_{\alpha_2}}\times \rE_{V_{\beta_1}}\times \rE_{V_{\beta_2}}&\square&\dTo_{\tilde{\kappa}_\lambda}\\
\dTo^{\tau_\lambda}&&\\
\rE_{V_{\alpha_1}}\times \rE_{V_{\beta_1}}\times \rE_{V_{\alpha_2}}\times \rE_{V_{\beta_2}} &\lTo^{{p_1}_\lambda} &E'_\lambda &\rTo^{{p_2}_\lambda} &E''_\lambda &\rTo^{{p_3}_\lambda}&\rE_{V_{\alpha'}}\times \rE_{V_{\beta'}}
\end{diagram}

For each $\lambda=(\alpha_1,\alpha_2,\beta_1,\beta_2)\in \cN$, we denote by 
$E'_\lambda=E'^{\alpha'}_{\alpha_1,\beta_1}\times E'^{\beta'}_{\alpha_2,\beta_2}, E''_\lambda=E''^{\alpha'}_{\alpha_1,\beta_1}\times E''^{\beta'}_{\alpha_2,\beta_2}$ and $\iota_\lambda=\iota^{\alpha}_{\alpha_1,\alpha_2}\times \iota^{\beta}_{\beta_1,\beta_2}, \kappa_\lambda=\kappa^{\alpha}_{\alpha_1,\alpha_2}\times \kappa^{\beta}_{\beta_1,\beta_2},{p_1}_\lambda={p_1}^{\alpha'}_{\alpha_1,\beta_1}\times {p_1}^{\beta'}_{\alpha_2,\beta_2},{p_2}_\lambda={p_2}^{\alpha'}_{\alpha_1,\beta_1}\times {p_2}^{\beta'}_{\alpha_2,\beta_2}$.

By definition, the right hand side is equal to 
\begin{align*}
\bigoplus_{\lambda=(\alpha_1,\alpha_2,\beta_1,\beta_2)\in \cN}({p_3}_\lambda)_!({p_2}_\lambda)_\flat({p_1}_\lambda)^*(\tau_\lambda)_!(\kappa_\lambda)_!(\iota_\lambda)^*(A\boxtimes B)[N_\lambda-(\alpha_2,\beta_1)](\frac{N_\lambda-(\alpha_2,\beta_1)}{2}),
\end{align*}
where 
$N_\lambda=-\langle\alpha_1,\alpha_2\rangle-\langle\beta_1,\beta_2\rangle+\sum_{h\in H}(\alpha_{1h'}\beta_{1h''}+\alpha_{2h'}\beta_{2h''})+\sum_{i\in I}(\alpha_{1i}\beta_{1i}+\alpha_{2i}\alpha_{2i})$
for any $\lambda=(\alpha_1,\alpha_2,\beta_1,\beta_2)\in \cN$.

For each $\lambda=(\alpha_1,\alpha_2,\beta_1,\beta_2)\in \cN$, let 
$$\cO_\lambda=(F^{\alpha}_{\alpha_1,\alpha_2}\times F^{\beta}_{\beta_1,\beta_2})\times_{\rE_{V_{\alpha_1}}\times \rE_{V_{\beta_1}}\times \rE_{V_{\alpha_2}}\times\rE_{V_{\beta_2}}}E'_\lambda$$
be the fibre product of $\tau_\lambda\kappa_\lambda:F^{\alpha}_{\alpha_1,\alpha_2}\times F^{\beta}_{\beta_1,\beta_2}\rightarrow \rE_{V_{\alpha_1}}\times \rE_{V_{\beta_1}}\times \rE_{V_{\alpha_2}}\times\rE_{V_{\beta_2}}$ and ${p_1}_\lambda:E'_\lambda\rightarrow \rE_{V_{\alpha_1}}\times \rE_{V_{\beta_1}}\times \rE_{V_{\alpha_2}}\times\rE_{V_{\beta_2}}$, that is, there is a Cartesian diagram
\begin{diagram}[midshaft,size=2em]
F^{\alpha}_{\alpha_1,\alpha_2}\times F^{\beta}_{\beta_1,\beta_2} &\lTo^{\tilde{p_1}_\lambda} &\cO_\lambda\\
\dTo^{\tau_\lambda\kappa_\lambda} &\Box &\dTo_{\tilde{\kappa}_\lambda}\\
\rE_{V_{\alpha_1}}\times \rE_{V_{\beta_1}}\times \rE_{V_{\alpha_2}}\times\rE_{V_{\beta_2}} &\lTo^{{p_1}_\lambda} &E'_\lambda.
\end{diagram}

By base change, we have $({p_1}_\lambda)^*(\tau_\lambda)_!(\kappa_\lambda)_!\simeq (\tilde{\kappa}_\lambda)_!(\tilde{p_1}_\lambda)^*$, and so we have following proposition.
\begin{proposition}\label{right}
The right hand side of the formula in Theorem \ref{main} is isomorphic to 
$$\bigoplus_{\lambda=(\alpha_1,\alpha_2,\beta_1,\beta_2)\in \cN}({p_3}_\lambda)_!({p_2}_\lambda)_\flat(\tilde{\kappa}_\lambda)_!(\tilde{p_1}_\lambda)^*(\iota_\lambda)^*(A\boxtimes B)[N_\lambda-(\alpha_2,\beta_1)](\frac{N_\lambda-(\alpha_2,\beta_1)}{2}).$$
\end{proposition}

\subsection{Connection between two sides}\

We will define a variety $\cP_\lambda$ and morphisms $\varphi_\lambda:\cP_\lambda\rightarrow \cQ_\lambda,\psi_\lambda:\cP_\lambda\rightarrow \cO_\lambda$ for any $\lambda=(\alpha,\alpha_2,\beta_1,\beta_2)\in \cN$.
\begin{diagram}[midshaft,size=2em]
\rE_{V_\alpha}\times \rE_{V_\beta} &&&\lTo^{{p_1}^{\gamma}_{\alpha,\beta}} &E'^{\gamma}_{\alpha,\beta}\\
 &&&&\uInto^{\tilde{\iota}_\lambda}\\
\uInto^{\iota_\lambda}&&& & \cQ_\lambda &\rTo^{\tilde{p_2}} &\tilde{F}_\lambda\\
& &&&\uDashto^{\varphi_\lambda}  & &\vEq& & & &\\
F^{\alpha}_{\alpha_1,\alpha_2}\times F^{\beta}_{\beta_1,\beta_2} &\lTo^{\tilde{p_1}_\lambda} &\cO_\lambda&\lDashto^{\psi_\lambda}&\cP_\lambda&\rTo^{\tilde{p_2}_\lambda}&\tilde{F}_\lambda\\
\dTo^{\kappa_\lambda}&&&&& &\\
\rE_{V_{\alpha_1}}\times \rE_{V_{\alpha_2}}\times \rE_{V_{\beta_1}}\times \rE_{V_{\beta_2}}&\square&\dTo_{\tilde{\kappa}_\lambda}&&\dTo^{\tilde{f}_\lambda}&\square&\dTo_{f_\lambda}\\
\dTo^{\tau_\lambda}&&&&\\
\rE_{V_{\alpha_1}}\times \rE_{V_{\beta_1}}\times \rE_{V_{\alpha_2}}\times \rE_{V_{\beta_2}} &\lTo^{{p_1}_\lambda} &E'_\lambda&\hEq &E'_\lambda &\rTo^{{p_2}_\lambda} &E''_\lambda 
\end{diagram}

Up to isomorphism, the definitions of induction and restriction functors are independent of the choices of $I$-graded vector space $V_\nu$ for each dimension vector $\nu\in \bbN I$, and the definition of $\Res^\nu_{\nu',\nu''}$ is independent of the choices of the fixed $I$-graded subspace $W^{\nu'}\subset V_\nu$ and the fixed $I$-graded linear isomorphisms $\rho_1^{\nu''}:V_\nu/W^{\nu''}\xrightarrow{\simeq} V_{\nu'},\rho_2^{\nu''}:W^{\nu''}\xrightarrow{\simeq}V_{\nu''}$. From now on, we assume that $V_{\alpha}=V_{\alpha_1}\bigoplus V_{\alpha_2}, V_{\beta}=V_{\beta_1}\bigoplus V_{\beta_2},V_{\alpha'}=V_{\alpha_1}\bigoplus V_{\beta_1}, V_{\beta'}=V_{\alpha_2}\bigoplus V_{\beta_2}, V_\gamma=V_\alpha\bigoplus V_\beta=V_{\alpha'}\bigoplus V_{\beta'}=V_{\alpha_1}\bigoplus V_{\alpha_2}\bigoplus V_{\beta_1}\bigoplus V_{\beta_2}$, and assume that the fixed $I$-graded subspaces $W^{\beta'}=V_{\beta'},W^{\alpha_2}=V_{\alpha_2}, W^{\beta_2}=V_{\beta_2}$, the fixed $I$-graded linear isomorphisms $\rho_1^{\beta'},\rho_2^{\beta'},\rho_1^{\alpha_2},\rho_2^{\alpha_2},\rho_1^{\beta_2},\rho_2^{\beta_2}$ are the identity maps. 

Note that the variety $\cQ_\lambda$ consists of $(x,W,\rho_1,\rho_2)$, where $x\in \rE_{V_\gamma}$ such that $V_{\beta'}$ is $x$-stable, $W\subset V_\gamma$ is a $x$-stable $I$-graded subspace of dimension vector $\beta$ such that the dimension vector of $W\cap V_{\beta'}$ is $\beta_2$, and $\rho_1:V_\gamma/W\xrightarrow{\simeq}V_\alpha, \rho_2:W\xrightarrow{\simeq}V_\beta$ are $I$-graded linear isomorphisms. 

The variety $\cO_\lambda$ consists of $(x_\alpha,x_\beta, (x_{\alpha'}, W_1, \rho_{11},\rho_{12}), (x_{\beta'}, W_2, \rho_{21},\rho_{22}))$, where $x_\alpha\in \rE_{V_\alpha},x_\beta\in \rE_{V_\beta},x_{\alpha'}\in\rE_{V_{\alpha'}},x_{\beta'}\in \rE_{V_{\beta'}}$ such that $V_{\alpha_2}$ is $x_\alpha$-stable and $V_{\beta_2}$ is $x_\beta$-stable, $W_1\subset V_{\alpha'}$ is a $x_{\alpha'}$-stable $I$-graded subspace of dimension vector $\beta_1$, $W_2\subset V_{\beta'}$ is a $x_{\beta'}$-stable $I$-graded subspace of dimension vector $\beta_2$, and $\rho_{11}:V_{\alpha'}/W_1\xrightarrow{\simeq} V_{\alpha_1},\rho_{12}:W_1\xrightarrow{\simeq} V_{\beta_1},\rho_{21}:V_{\beta'}/W_2\xrightarrow{\simeq} V_{\alpha_2},\rho_{22}:W_2\xrightarrow{\simeq} V_{\beta_2}$ are $I$-graded linear isomorphisms satisfying
\begin{align*}
\overline{x_\alpha}^{V_{\alpha_2}}=(\rho_{11})_*\overline{x_{\alpha'}}^{W_1},\ &\overline{x_\beta}^{V_{\beta_2}}=(\rho_{12})_*x_{\alpha'}|_{W_1},\\
x_\alpha|_{V_{\alpha_2}}=(\rho_{21})_*\overline{x_{\beta'}}^{W_2},\ &x_\beta|_{V_{\beta_2}}=(\rho_{22})_*x_{\beta'}|_{W_2}.
\end{align*}

Let 
$$\cP_\lambda=\tilde{F}_\lambda\times_{E''_\lambda}E'_\lambda$$
be the fibre product of $f_\lambda:\tilde{F}_\lambda\rightarrow E''_\lambda$ and ${p_2}_\lambda:E'_\lambda\rightarrow E''_\lambda$, that is, there is a Cartesian diagram
\begin{diagram}[midshaft,size=2em]
\cP_\lambda &\rTo^{\tilde{p_2}_\lambda} &\tilde{F}_\lambda\\
\dTo^{\tilde{f}_\lambda} &\Box &\dTo_{f_\lambda}\\
E'_\lambda &\rTo^{{p_2}_\lambda} &E''_\lambda.
\end{diagram}

Note that the variety $\cP_\lambda$ consists of $(x,W,\rho_{11},\rho_{12},\rho_{21},\rho_{22})$, where $x\in \rE_{V_\gamma}$ such that $V_{\beta'}$ is $x$-stable, $W\subset V_\gamma$ is a $x$-stable $I$ -graded subspace of dimension vector $\beta$ such that the dimension vector of $W\cap V_{\beta'}$ is $\beta_2$, and $\rho_{11}:V_{\alpha'}/(W/W\cap V_{\beta'})\xrightarrow{\simeq}V_{\alpha_1},\rho_{12}:W/W\cap V_{\beta'}\xrightarrow{\simeq}V_{\beta_1},\rho_{21}:V_{\beta'}/W\cap V_{\beta'}\xrightarrow{\simeq}V_{\alpha_2},\rho_{22}:W\cap V_{\beta'}\xrightarrow{\simeq}V_{\beta_2}$ are $I$-graded linear isomorphisms, where we regard $W/W\cap V_{\beta'}\simeq W+V_{\beta'}/V_{\beta'}$ as a subspace of $V_\gamma/V_{\beta'}=V_{\alpha'}$. Note that the morphism $\tilde{p_2}_\lambda:\cP_\lambda\rightarrow \tilde{F}_\lambda$ is a principal $\rG_{V_{\alpha_1}}\times \rG_{V_{\beta_1}}\times \rG_{V_{\alpha_2}}\times \rG_{V_{\beta_2}}$-bundle.

\begin{lemma}\label{f-p2}
We have ${f_\lambda}_!(\tilde{p_2}_\lambda)_\flat\simeq ({p_2}_\lambda)_\flat(\tilde{f}_\lambda)_!$.
\end{lemma}
\begin{proof}
By base change of above Cartesian diagram, we have $({p_2}_\lambda)^*(f_\lambda)_!\simeq ({\tilde{f}_\lambda})_!(\tilde{p_2}_\lambda)^*$. By Proposition \ref{Lusztig-principal}, $(\tilde{p_2}_\lambda)_\flat, ({p_2}_\lambda)_\flat$ have quasi-inverses $(\tilde{p_2}_\lambda)^*,({p_2}_\lambda)^*$ respectively, so $({f_\lambda})_!(\tilde{p_2}_\lambda)_\flat\simeq ({p_2}_\lambda)_\flat(\tilde{f}_\lambda)_!$.
\end{proof}

\begin{lemma}\label{psi}
There is a smooth morphism $\psi_\lambda:\cP_\lambda\rightarrow \cO_\lambda$ with connected fibres of dimension
$$K_\lambda=L_\lambda-\sum_{h\in H}(\alpha_{1h'}\alpha_{2h''}+\beta_{1h'}\beta_{2h''})$$
such that $\tilde{f}_\lambda=\tilde{\kappa}_\lambda\psi_\lambda$.
\end{lemma}
\begin{proof}
For any $p=(x,W,\rho_{11},\rho_{12},\rho_{21},\rho_{22})\in \cP_\lambda$, we need to define an element $\psi_\lambda(p)=(x_\alpha,x_\beta, (x_{\alpha'}, W_1, \rho_{11},\rho_{12}), (x_{\beta'}, W_2, \rho_{21},\rho_{22}))\in \cO_\lambda$. 

We define $(x_{\alpha'},W_1,x_{\beta'},W_2)=f_\lambda(x,W)$, that is
\begin{align*}
x_{\alpha'}=\overline{x}^{V_{\beta'}}, x_{\beta'}=x|_{V_{\beta'}}, W_1=W/W\cap V_{\beta'}, W_2=W\cap V_{\beta'}.
\end{align*}
and $\rho_{11},\rho_{12},\rho_{21},\rho_{22}$ are the same as them in $p$.
It remains to define $(x_{\alpha},x_{\beta})\in F^{\alpha}_{\alpha_1,\alpha_2}$. 

Note that 
$$\tau_\lambda\kappa_\lambda:F^{\alpha}_{\alpha_1,\alpha_2}\times F^{\beta}_{\beta_1,\beta_2}\rightarrow \rE_{V_{\alpha_1}}\times \rE_{V_{\beta_1}}\times \rE_{V_{\alpha_2}}\times \rE_{V_{\beta_2}}$$
is a vector bundle whose fibres can be identified with 
$$\bigoplus_{h\in H}(\Hom_k((V_{\alpha_1})_{h'},(V_{\alpha_2})_{h''})\bigoplus \Hom_k((V_{\beta_1})_{h'},(V_{\beta_2})_{h''})).$$
By the definition of $\cO_\lambda$, the image
$$\tau_\lambda\kappa_\lambda(x_\alpha,x_\beta)={p_1}_\lambda((x_{\alpha'}, W_1, \rho_{11},\rho_{12}), (x_{\beta'}, W_2, \rho_{21},\rho_{22}))$$ is already determined, and so it remains to define an element $(y_1,y_2)$ in the fibre $$\bigoplus_{h\in H}(\Hom_k((V_{\alpha_1})_{h'},(V_{\alpha_2})_{h''})\bigoplus \Hom_k((V_{\beta_1})_{h'},(V_{\beta_2})_{h''})).$$

By the proof of Lemma \ref{vector bundle}, the fibre of the vector bundle 
$$f_\lambda:\tilde{F}_\lambda\rightarrow E''_\lambda$$
at $((x_{\alpha'},W_1),(x_{\beta'},W_2))\in E''_\lambda$ can be identified with the linear space of pairs $(z,y)$, where $z:W_1\rightarrow V_{\beta'}/W_2$ is a $I$-graded linear map and $y\in \bigoplus_{h\in H}\Hom_k((V_{\alpha'})_{h'},(V_{\beta'})_{h''})$ satisfying certain conditions. Thus, from $p=(x,W,\rho_{11},\rho_{12},\rho_{21},\rho_{22})\in \cP_\lambda$, we have $(x,W)\in \tilde{F}_\lambda$ which can be identified with $(x_{\alpha'},W_1,x_{\beta'},W_2,z,y)$, where $y\in \bigoplus_{h\in H}\Hom_k((V_{\alpha'})_{h'},(V_{\beta'})_{h''})$ can be written as
$$\begin{pmatrix}
y_{11}&y_{12}\\
y_{21}&y_{22}
\end{pmatrix}\in \bigoplus_{h\in H}\begin{pmatrix}
\Hom_k((V_{\alpha_1})_{h'},(V_{\alpha_2})_{h''})&\Hom_k((V_{\alpha_1})_{h'},(V_{\beta_2})_{h''})\\
\Hom_k((V_{\beta_1})_{h'},(V_{\alpha_2})_{h''})&\Hom_k((V_{\beta_1})_{h'},(V_{\beta_2})_{h''})\end{pmatrix}.$$

Hence we may define $y_1=y_{11},y_2=y_{22}$.

Since ${p_1}_\lambda$ is smooth with connected fibres, $f_\lambda$ is a vector bundle and $y\mapsto (y_1,y_2)$ is a projection, we know that $\psi_\lambda$ is smooth with connected fibres.

It is clear that $\tilde{f}_\lambda=\tilde{\kappa}_\lambda\psi_\lambda$, and the fibres of $\tilde{f}_\lambda$ have the same dimension as the fibres of $f_\lambda$, that is $L_\lambda$, while the fibres of $\tilde{\kappa}_\lambda$ have the same dimension as the fibres of $\tau_\lambda\kappa_\lambda$, that is $\sum_{h\in H}(\alpha_{1h'}\alpha_{2h''}+\beta_{1h'}\beta_{2h''})$. Therefore, $K_\lambda=L_\lambda-\sum_{h\in H}(\alpha_{1h'}\alpha_{2h''}+\beta_{1h'}\beta_{2h''})$.
\end{proof}

The morphism $\iota_\lambda\tilde{p_1}_\lambda\psi_\lambda$ can be written as
\begin{align*}
(x,W,\rho_{11},\rho_{12},\rho_{21},\rho_{22})=(x_{\alpha'},W_1,x_{\beta'},W_2,z,y,\rho_{11},\rho_{12},\rho_{21},\rho_{22})\mapsto \\
(\begin{pmatrix}
\rho_{11}\overline{x_{\alpha'}}^{W_1}\rho_{11}^{-1} &0\\
y_{11} &\rho_{21}\overline{x_{\beta'}}^{W_2}\rho_{21}^{-1}
\end{pmatrix},\begin{pmatrix}
\rho_{12}x_{\alpha'}|_{W_1}\rho_{12}^{-1} &0\\
y_{22} &\rho_{22}x_{\beta'}|_{W_2}\rho_{22}^{-1}
\end{pmatrix}),
\end{align*}
where we write elements in $\rE_{V_\alpha},\rE_{V_\beta}$ as block matrices with respect to the direct sum decompositions $V_\alpha=V_{\alpha_1}\bigoplus V_{\alpha_2}, V_\beta=V_{\beta_1}\bigoplus V_{\beta_2}$ respectively. 

We want to define a morphism $\varphi_\lambda:\cP_\lambda\rightarrow \cQ_\lambda$ such that ${p_1}^\gamma_{\alpha,\beta}\tilde{\iota}_\lambda\varphi_\lambda=\iota_\lambda\tilde{p_1}_\lambda\psi_\lambda$ and $\tilde{p_2}_\lambda=\tilde{p_2}\varphi_\lambda$. Note that $\tilde{p_2}:\cQ_\lambda\rightarrow \tilde{F}_\lambda$ is a principal $\rG_{V_\alpha}\times \rG_{V_\beta}$-bundle, $\tilde{p_2}_\lambda:\cP_\lambda\rightarrow \tilde{F}_\lambda$ is a principal $\rG_{V_{\alpha_1}}\times \rG_{V_{\beta_1}}\times \rG_{V_{\alpha_2}}\times \rG_{V_{\beta_2}}$-bundle, and $\rG_{V_{\alpha_1}}\times \rG_{V_{\beta_1}}\times \rG_{V_{\alpha_2}}\times \rG_{V_{\beta_2}}$ can be embedded into $\rG_{V_\alpha}\times \rG_{V_\beta}$.

\begin{lemma}\label{varphi}
There is a morphism $\varphi_\lambda:\cP_\lambda\rightarrow \cQ_\lambda$ such that ${p_1}^\gamma_{\alpha,\beta}\tilde{\iota}_\lambda\varphi_\lambda=\iota_\lambda\tilde{p_1}_\lambda\psi_\lambda$ and $\tilde{p_2}_\lambda=\tilde{p_2}\varphi_\lambda$, hence we have $(\varphi_\lambda)^*(\tilde{\iota}_\lambda)^*({p_1}^\gamma_{\alpha,\beta})^*=(\psi_\lambda)^*(\tilde{p_1}_\lambda)^*(\iota_\lambda)^*$ and $(\tilde{p_2})_\flat\simeq(\tilde{p_2}_\lambda)_\flat(\varphi_\lambda)^*$.
\end{lemma}
\begin{proof}
For any $p=(x,W,\rho_{11},\rho_{12},\rho_{21},\rho_{22})\in \cP_\lambda$, we need to define an element $\varphi(p)=(x,W,\rho_1,\rho_2)\in \cQ_\lambda$. We define $x,W$ are the same as them in $p$.

It remains to define $\rho_1:V_\gamma/W\xrightarrow{\simeq}V_\alpha=V_{\alpha_1}\bigoplus V_{\alpha_2}$ and $\rho_2:W\xrightarrow{\simeq}V_{\beta_1}\oplus V_{\beta_2}$. By the proof of Lemma \ref{vector bundle}, from $p=(x,W,\rho_{11},\rho_{12},\rho_{21},\rho_{22})\in \cP_\lambda$, we have $(x,W)\in \tilde{F}_\lambda$ which can be identified with $(x_{\alpha'},W_1,x_{\beta'},W_2,z,y)$ satisfying certain conditions via
$$x=\begin{pmatrix}
x_{\alpha'} &0\\
y &x_{\beta'}\end{pmatrix},
W=\{v'+v''\in V_{\alpha'}\oplus V_{\beta'}\mid v'\in W_1,z(v')=v''+W_2\}.$$
Under this correspondence, the $I$-graded linear isomorphisms in $p$ are $\rho_{11}:V_{\alpha'}/W_1\xrightarrow{\simeq}V_{\alpha_1},\rho_{12}:W_1\xrightarrow{\simeq}V_{\beta_1},\rho_{21}:V_{\beta'}/W_2\xrightarrow{\simeq}V_{\alpha_2},\rho_{22}:W_2\xrightarrow{\simeq}V_{\beta_2}$.

We fix $I$-graded linear isomorphisms $\zeta_z:V_\gamma/W\xrightarrow{\simeq} V_\alpha'/W_1\bigoplus V_{\beta'}/W_2$ and $\xi_z:W\xrightarrow{\simeq} W_1\bigoplus W_2$, then we define
$$\rho_1=\begin{pmatrix}
\rho_{11} &0\\
0 &\rho_{21}
\end{pmatrix}\zeta_z,\ \rho_2=\begin{pmatrix}
\rho_{12} &0\\
0 &\rho_{22}
\end{pmatrix}\xi_z.$$

It is routine to check ${p_1}^\gamma_{\alpha,\beta}\tilde{\iota}_\lambda\varphi_\lambda=\iota_\lambda\tilde{p_1}_\lambda\psi_\lambda$. We write elements in $\rE_{V_{\gamma}}$ as block matrices, recall that $V_\gamma=V_{\alpha_1}\bigoplus V_{\beta_1}\bigoplus V_{\alpha_2}\bigoplus V_{\beta_2}$.
Suppose  
$$y=\begin{pmatrix}
y_{11}&y_{12}\\
y_{21}&y_{22}
\end{pmatrix}\in \bigoplus_{h\in H}\begin{pmatrix}
\Hom_k((V_{\alpha_1})_{h'},(V_{\alpha_2})_{h''})&\Hom_k((V_{\alpha_1})_{h'},(V_{\beta_2})_{h''})\\
\Hom_k((V_{\beta_1})_{h'},(V_{\alpha_2})_{h''})&\Hom_k((V_{\beta_1})_{h'},(V_{\beta_2})_{h''})\end{pmatrix},$$
and $x_{\alpha'}=\begin{pmatrix}\overline{x_{\alpha'}}^{W_1} &(x_{\alpha'})_{12}\\
(x_{\alpha'})_{21} &x_{\alpha'}|_{W_1}\end{pmatrix},x_{\beta'}=\begin{pmatrix}
\overline{x_{\beta'}}^{W_2} &(x_{\beta'})_{12}\\
(x_{\beta'})_{21} &x_{\beta'}|_{W_2}
\end{pmatrix}$,
then $x\in \rE_{V_{\gamma}}$ corresponds to a block matrix 
$$\begin{pmatrix}
\overline{x_{\alpha'}}^{W_1} &(x_{\alpha'})_{12} &0 &0\\
(x_{\alpha'})_{21} &x_{\alpha'}|_{W_1} &0 &0\\
\rho_{21}^{-1}y_{11}\rho_{11} &\rho_{22}^{-1}y_{12}\rho_{11} &\overline{x_{\beta'}}^{W_2} &(x_{\beta'})_{12}\\
\rho_{21}^{-1}y_{21}\rho_{12} &\rho_{22}^{-1}y_{22}\rho_{21} &(x_{\beta'})_{21} &x_{\beta'}|_{W_2}
\end{pmatrix}$$
under the isomorphism $\textrm{Diag}(\rho_{11}^{-1},\rho_{12}^{-1},\rho_{21}^{-1},\rho_{22}^{-1}):V_\gamma\xrightarrow{\simeq} V_{\alpha'}/W_1\bigoplus W_1\bigoplus V_{\beta'}/W_2\bigoplus W_2$, and it can be rewritten as 
$$\begin{pmatrix}
\overline{x_{\alpha'}}^{W_1} &0 &(x_{\alpha'})_{12} &0\\
\rho_{21}^{-1}y_{11}\rho_{11} &\overline{x_{\beta'}}^{W_2} &\rho_{22}^{-1}y_{12}\rho_{11} &(x_{\beta'})_{12}\\
(x_{\alpha'})_{21} &0 &x_{\alpha'}|_{W_1} &0\\
\rho_{21}^{-1}y_{21}\rho_{12} &(x_{\beta'})_{21} &\rho_{22}^{-1}y_{22}\rho_{21} &x_{\beta'}|_{W_2}
\end{pmatrix}$$
under the isomorphism $\textrm{Diag}(\rho_{11}^{-1},\rho_{21}^{-1},\rho_{12}^{-1},\rho_{22}^{-1}): V_\gamma\xrightarrow{\simeq} V_{\alpha'}/W_1\bigoplus V_{\beta'}/W_2\bigoplus W_1\bigoplus W_2$, and so $\overline{x}^W,x|_W$ can be written as 
$$\zeta_z^{-1}\begin{pmatrix}
\overline{x_{\alpha'}}^{W_1} &0\\
\rho_{21}^{-1}y_{11}\rho_{11} &\overline{x_{\beta'}}^{W_2}
\end{pmatrix}\zeta_z,\ \xi_z^{-1}\begin{pmatrix}
x_{\alpha'}|_{W_1} &0\\
\rho_{22}^{-1}y_{22}\rho_{21} &x_{\beta'}|_{W_2}
\end{pmatrix}\xi_z,
$$
thus 
\begin{align*}
&{p_1}^\gamma_{\alpha,\beta}\tilde{\iota}_\lambda\varphi_\lambda(p)=(\rho_1\overline{x}^W\rho_1^{-1},\rho_2x|_W\rho_2^{-1})\\
&=(\begin{pmatrix}\rho_{11}\overline{x_{\alpha'}}^{W_1}\rho_{11}^{-1} &0\\
y_{11} &\rho_{21}\overline{x_{\beta'}}^{W_2}\rho_{21}^{-1}
\end{pmatrix},\begin{pmatrix}
\rho_{12}x_{\alpha'}|_{W_1}\rho_{12}^{-1} &0\\
y_{22} &\rho_{22}x_{\beta'}|_{W_2}\rho_{22}^{-1}
\end{pmatrix})=\iota_\lambda\tilde{p_1}_\lambda\psi_\lambda(p)
\end{align*}

It is clear that $\tilde{p_2}_\lambda=\tilde{p_2}\varphi_\lambda$, thus we have 
$(\varphi_\lambda)^*(\tilde{\iota}_\lambda)^*({p_1}^\gamma_{\alpha,\beta})^*=(\psi_\lambda)^*(\tilde{p_1}_\lambda)^*(\iota_\lambda)^*$ and $(\tilde{p_2}_\lambda)^*=(\varphi_\lambda)^*(\tilde{p_2})^*$. By Proposition \ref{Lusztig-principal}, $(\tilde{p_2}_\lambda)^*,(\tilde{p_2})^*$ have quasi-inverses $(\tilde{p_2}_\lambda)_\flat,(\tilde{p_2})_\flat$ respectively, and so $(\tilde{p_2})_\flat\simeq(\tilde{p_2}_\lambda)_\flat(\varphi_\lambda)^*$.
\end{proof}

\subsection{Proof of the main theorem}\

Now, we can prove Theorem \ref{main}.

\begin{proof}
By Proposition \ref{left} and Proposition \ref{right}, we only need to compare 
\begin{align*}
&({p_3}_\lambda)_!(f_\lambda)_!(\tilde{p_2})_\flat(\tilde{\iota}_\lambda)^*({p_1}^{\gamma}_{\alpha,\beta})^*(A\boxtimes B)[M](\frac{M}{2}),\\
&({p_3}_\lambda)_!({p_2}_\lambda)_\flat(\tilde{\kappa}_\lambda)_!(\tilde{p_1}_\lambda)^*(\iota_\lambda)^*(A\boxtimes B)[N_\lambda-(\alpha_2,\beta_1)](\frac{N_\lambda-(\alpha_2,\beta_1)}{2})
\end{align*}
for any $\lambda=(\alpha,\alpha_2,\beta_1,\beta_2)\in \cN$. By Lemma \ref{f-p2}, \ref{psi} and \ref{varphi}, we have
\begin{align*}
&({p_3}_\lambda)_!(f_\lambda)_!(\tilde{p_2})_\flat(\tilde{\iota}_\lambda)^*({p_1}^{\gamma}_{\alpha,\beta})^*(A\boxtimes B)[M](\frac{M}{2})\\
\simeq&({p_3}_\lambda)_!(f_\lambda)_!(\tilde{p_2}_\lambda)_\flat(\varphi_\lambda)^*(\tilde{\iota}_\lambda)^*({p_1}^{\gamma}_{\alpha,\beta})^*(A\boxtimes B)[M](\frac{M}{2})\\
\simeq&({p_3}_\lambda)_!({p_2}_\lambda)_\flat(\tilde{f}_\lambda)_!(\varphi_\lambda)^*(\tilde{\iota}_\lambda)^*({p_1}^{\gamma}_{\alpha,\beta})^*(A\boxtimes B)[M](\frac{M}{2})\\
\simeq&({p_3}_\lambda)_!({p_2}_\lambda)_\flat(\tilde{\kappa}_\lambda)_!(\psi_\lambda)_!(\varphi_\lambda)^*(\tilde{\iota}_\lambda)^*({p_1}^{\gamma}_{\alpha,\beta})^*(A\boxtimes B)[M](\frac{M}{2})\\
=&({p_3}_\lambda)_!({p_2}_\lambda)_\flat(\tilde{\kappa}_\lambda)_!(\psi_\lambda)_!(\psi_\lambda)^*(\tilde{p_1}_\lambda)^*(\iota_\lambda)^*(A\boxtimes B)[M](\frac{M}{2}).
\end{align*}
Since $\psi_\lambda$ is smooth with connected fibres of dimension $K_\lambda$, by Proposition \ref{pure-*!}, we have $(\psi_\lambda)^*=(\psi_\lambda)^![-2K_\lambda](-K_\lambda)$. Hence $$(\psi_\lambda)_!(\psi_\lambda)^*=(\psi_\lambda)_!(\psi_\lambda)^![-2K_\lambda](-K_\lambda)\simeq [-2K_\lambda](-K_\lambda),$$ since $((\psi_\lambda)_!,(\psi_\lambda)^!)$ is an adjoint pair and $(\psi_\lambda)^!=(\psi_\lambda)^*[2K_\lambda](K_\lambda)$ is fully faithful, by Theorem 3.6.6 in \cite{Pramod-2021}. Thus 
\begin{align*}
&({p_3}_\lambda)_!({p_2}_\lambda)_\flat(\tilde{\kappa}_\lambda)_!(\psi_\lambda)_!(\psi_\lambda)^*(\tilde{p_1}_\lambda)^*(\iota_\lambda)^*(A\boxtimes B)[M](\frac{M}{2})\\
\simeq&({p_3}_\lambda)_!({p_2}_\lambda)_\flat(\tilde{\kappa}_\lambda)_!(\tilde{p_1}_\lambda)^*(\iota_\lambda)^*(A\boxtimes B)[M-2K_\lambda](\frac{M-2K_\lambda}{2}).
\end{align*}
It remains to show that $M-2K_\lambda=N_\lambda-(\alpha_2,\beta_1)$. Indeed, 
\begin{align*}
&M-N_\lambda-2K_\lambda\\
=&\sum_{h\in H}\alpha_{h'}\beta_{h''}+\sum_{i\in I}\alpha_i\beta_i-\langle\alpha',\beta'\rangle+\langle\alpha_1,\alpha_2\rangle+\langle\beta_1,\beta_2\rangle-\sum_{h\in H}(\alpha_{1h'}\beta_{1h''}+\alpha_{2h'}\beta_{2h''})\\
&-\sum_{i\in I}(\alpha_{1i}\beta_{1i}+\alpha_{2i}\alpha_{2i})-2\sum_{h\in H}(\alpha_{1h'}\alpha_{2h''}+\alpha_{1h'}\beta_{2h''}+\beta_{1h'}\beta_{2h''})-2\sum_{i\in I}\alpha_{2i}\beta_{1i}\\
&+2\sum_{h\in H}(\alpha_{1h'}\alpha_{2h''}+\beta_{1h'}\beta_{2h''})\\
=&\sum_{h\in H}(\alpha_{1h'}+\alpha_{2h'})(\beta_{1h''}+\beta_{2h''})+\sum_{i\in I}(\alpha_{1i}+\alpha_{2i})(\beta_{1i}+\beta_{2i})-\langle\alpha_1+\beta_1,\alpha_2+\beta_2\rangle\\
&+\langle\alpha_1,\alpha_2\rangle+\langle\beta_1,\beta_2\rangle-\sum_{h\in H}(\alpha_{1h'}\beta_{1h''}+\alpha_{2h'}\beta_{2h''})-\sum_{i\in I}(\alpha_{1i}\beta_{1i}+\alpha_{2i}\alpha_{2i})\\
&-2\sum_{h\in H}(\alpha_{1h'}\alpha_{2h''}+\alpha_{1h'}\beta_{2h''}+\beta_{1h'}\beta_{2h''})-2\sum_{i\in I}\alpha_{2i}\beta_{1i}+2\sum_{h\in H}(\alpha_{1h'}\alpha_{2h''}+\beta_{1h'}\beta_{2h''})\\
=&\sum_{h\in H}(\alpha_{2h'}\beta_{1h''}-\alpha_{1h'}\beta_{2h''})+\sum_{i\in I}(\alpha_{1i}\beta_{2i}-\alpha_{2i}\beta_{1i})-\langle\alpha_1,\beta_2\rangle-\langle\beta_1,\alpha_2\rangle\\
=&\sum_{h\in H}(\alpha_{2h'}\beta_{1h''}+\beta_{1h'}\alpha_{2h''})-\sum_{i\in I}(\alpha_{2i}\beta_{1i}+\beta_{1i}\alpha_{2i})\\
=&-(\alpha_2,\beta_1),
\end{align*}
as desired.
\end{proof}

\textbf{Acknowledgements}

We are grateful to the referee for many helpful suggestions and corrections.

\nocite{Lusztig-1990}
\nocite{Schiffmann-2012}

\bibliography{mybibfile}

\end{document}